\documentclass[a4paper,11pt,reqno,noindent]{amsart}
\usepackage[centertags]{amsmath}
\usepackage{amsfonts,amssymb,amsthm,amscd,dsfont,cases,esint,enumerate,newlfont,color}
\usepackage[T1]{fontenc}
\usepackage[english]{babel}
\usepackage[applemac]{inputenc}
\usepackage[body={15cm,21.5cm},centering]{geometry} 
\usepackage{fancyhdr}
\theoremstyle{plain}
\newtheorem{theor10}{Theorem}
\newenvironment{theor1}
  {\pushQED{\qed}\begin{theor10}}
  {\popQED\end{theor10}}
\newtheorem{prop10}{Proposition}

\newtheorem{theor0}{Theorem}[section]

\newtheorem{lem0}[theor0]{Lemma}

\newtheorem{prop0}[theor0]{Proposition}
\newenvironment{prop}
  {\pushQED{\qed}\begin{prop0}}
  {\popQED\end{prop0}}
\theoremstyle{definition}
\newtheorem{defin0}[theor0]{Definition}

\newtheorem{rems0}[theor0]{Remarks}

\newtheorem{rem0}[theor0]{Remark}
\newenvironment{rem}
  {\pushQED{\qed}\begin{rem0}}
  {\popQED\end{rem0}}

\mathchardef\emptyset="001F
\numberwithin{equation}{section}

\newcommand{\e}{\varepsilon}

\newcommand{\Lc}{\mathcal{L}}
\newcommand{\Cc}{\mathcal{C}}
\newcommand{\calC}{\mathcal{C}}
\newcommand{\Kc}{\mathcal K}
\newcommand{\Ec}{\mathcal E}
\newcommand{\R}{\mathbb R}
\newcommand{\Ic}{\mathcal I}
\newcommand{\Mk}{\mathfrak M}
\newcommand{\Md}{\mathbb M}
\newcommand{\Nc}{\mathcal N}
\newcommand{\cvf}{\rightharpoonup}
\newcommand{\w}{\omega}
\newcommand{\loc}{{\operatorname{loc}}}
\newcommand{\Id}{\operatorname{Id}}
\newcommand{\E}{\mathbb{E}}

\newcommand{\dist}{\operatorname{dist}}
\newcommand{\D}{\operatorname{D}}

\newcommand{\bb}{\bar{\boldsymbol b}}
\newcommand{\Bb}{\bar{\boldsymbol B}}
\newcommand{\adh}{{\operatorname{adh}}}
\newcommand{\Ld}{\operatorname{L}}
\newcommand{\Div}{\operatorname{div}}
\newcommand{\Sym}{{\operatorname{sym}}}
\newcommand{\Skew}{{\operatorname{skew}}}
\newcommand{\Tr}{\operatorname{tr}}
\newcommand{\step}[1]{\noindent \textit{Step} #1.}
\newcommand{\substep}[1]{\noindent \textit{Substep} #1.}
\newcommand{\Pm}{\mathbb{P}}
\newcommand{\pr}[1]{\mathbb{P}\left[ #1 \right]}

\newcommand{\expec}[1]{\mathbb{E}\left[ #1 \right]}
\newcommand{\expecm}[1]{\mathbb{E}\big[ #1 \big]}

\usepackage[colorlinks,citecolor=black,urlcolor=black]{hyperref}

\title[Effective viscosity of colloidal suspensions]{Corrector equations in fluid mechanics:\\Effective viscosity of colloidal suspensions}

\author[M. Duerinckx]{Mitia Duerinckx}
\author[A. Gloria]{Antoine Gloria}
\address[Mitia Duerinckx]{Universit\'e Paris-Saclay, CNRS, Laboratoire de Math\'ematiques d'Orsay, 91405 Orsay, France \& Universit\'e Libre de Bruxelles, Département de Mathématique, 1050 Brussels, Belgium}
\email{mduerinc@ulb.ac.be}
\address[Antoine Gloria]{Sorbonne Universit\'e, CNRS, Universit\'e de Paris, Laboratoire Jacques-Louis Lions, 75005 Paris, France \& Universit\'e Libre de Bruxelles, Département de Mathématique, 1050 Brussels, Belgium}
\email{gloria@ljll.math.upmc.fr}

\begin{document}
\selectlanguage{english}

\maketitle

\begin{abstract}
Consider a colloidal suspension of rigid particles in a steady Stokes flow.
In a celebrated work, Einstein argued that in the regime of dilute particles the system behaves at leading order like a Stokes fluid 
with some explicit effective viscosity.
In the present contribution, we rigorously define a notion of effective viscosity, regardless of the
dilute regime assumption. More precisely, we establish a homogenization result when particles are distributed according to a given stationary and ergodic random point process. The main novelty is the introduction and analysis of suitable corrector equations.
\end{abstract}

\setcounter{tocdepth}{1}
\tableofcontents

\section{Introduction and main results}

\subsection{General overview}

This article is devoted to the large-scale behavior of the steady Stokes equation for a fluid with a dense colloidal suspension of small rigid particles that are randomly distributed. The fluid and the particles interact via the action-reaction principle, and satisfy a no-slip condition
at the particle boundaries.
Suspended particles then act as obstacles, hindering the fluid flow and therefore increasing the flow resistance, that is, the viscosity. The system is naturally expected to behave on large scales approximately like a Stokes fluid with some effective viscosity.
Our main result in this contribution makes this statement precise and rigorously defines the effective viscosity in terms of a stochastic homogenization result.

\medskip\noindent
Let us first describe previous contributions on the topic, and emphasize our main motivation.
In his PhD thesis, Einstein~\cite{Einstein-06} was the first to analyze this effective viscosity problem: focussing on a dilute regime (that is, assuming that particles are scarce), he argued that the fluid indeed behaves at leading order like a Stokes fluid with some effective viscosity and that the latter can be explicitly computed at first order in the particle concentration in form of the so-called Einstein's formula, which played a key role in the physics community at that time as it served as a basis for Perrin's celebrated experiment to estimate the Avogadro number.
Various contributions followed, in particular going beyond the first order, e.g.~\cite{BG-72,Nunan-Keller-84,Almog-Brenner,AGKL-12}.
From a rigorous perspective, several recent contributions stand out. In~\cite{Haines-Mazzu} (see also the refined version \cite{Niethammer-Schubert-19}), Haines and Mazzucato provide
bounds on the difference between a heuristic notion of effective viscosity (defined as some integral ratio with the correct dimensionality)
and Einstein's formula. In~\cite{GVH} (see also~\cite{GVM}), G\'erard-Varet and Hillairet took another approach, considering the solution of the Stokes problem and proving its closeness to the Stokes flow associated with some effective viscosity (described at higher order than Einstein's formula) --- a quantified consistency result. In both works, for the effective behavior of a sequence of solutions, the authors establish error estimates that only get sharp in the dilute regime. 
On the one hand, the analysis in~\cite{Haines-Mazzu,Niethammer-Schubert-19,GVH,GVM} requires sophisticated arguments (reflection method, renormalized energy method, etc.) in order to get quantitative statements. On the other hand, their applicability is limited by the dilute regime assumption that allows to construct ``explicit'' approximate solutions.
In particular, the very notion of effective viscosity is not defined independently of the dilute regime.
Our main motivation is to remedy this issue by taking yet another approach and distinguishing two independent questions:
\begin{itemize}
\item the definition of an effective viscosity in full generality in the setting of homogenization theory in terms of a suitable corrector problem;
\smallskip\item the asymptotic analysis of the effective viscosity in the dilute regime --- in the spirit of the so-called Clausius-Mossotti formula for homogenization of electrostatics and linear elasticity, cf.~\cite{DG-16a}.
\end{itemize}
The present contribution answers the first question, while the second one is the object of a forthcoming work~\cite{DG-20c}.

\medskip\noindent
In a nutshell, our approach is in the pure tradition of homogenization theory.
We reformulate the problem as the study of a family of solutions of fluid mechanics equations in a perforated domain associated with the spatial rescaling of some stationary and ergodic random array of inclusions, and we prove that this family converges to the solution of some effective (deterministic) fluid mechanics equation.
Periodic homogenization in fluid mechanics is not new, dating back to S\'anchez-Palencia~\cite{Sanchez}, Tartar~\cite{Tartar-Stokes}, and Allaire~\cite{Allaire-89,Allaire-90}, to cite but a few.
We also refer to the early work of Cioranescu and Saint Jean Paulin~\cite{Cio-SJP}, where a related scalar problem is considered in form of the so-called torsion problem.
In the random setting, we refer to the contributions by Beliaev and Kozlov~\cite{Beliaev-Kozlov-96}, by Basson and G\'erard-Varet~\cite{BGV-08}, and more recently by Giunti and H\"ofer~\cite{GiuntiHofer}. We further refer to the works of Jikov~\cite{Jikov-87,Jikov-90} on the closely related homogenization problem for stiff inclusions in linear elasticity, see also~\cite[Chapter~8.6]{JKO94}.
In the present work, the homogenization result that is established in the general stationary and ergodic random framework (independently of the dilute regime) is new even in the periodic setting due to the specificity of the considered boundary conditions.

\medskip
\noindent
In terms of insight, the main novelty of this contribution is the introduction and analysis of suitable corrector equations
in a context where this had not been done before.
From a mathematical perspective, the divergence-free constraint for the fluid velocity yields technical difficulties and makes the analysis quite subtle --- although still solely based on soft,  qualitative arguments.
As usual, the proof of the homogenization result splits into two parts: the construction of correctors, 
and the convergence result using Tartar's method of oscillating test functions~\cite{Tartar-09}.
The development of a corresponding {\it quantitative} homogenization theory, which is postponed to a forthcoming work~\cite{DG-20a}, requires a suitable strong mixing condition on the particle distribution.

\medskip
\noindent
Before turning to the actual statement of the main results, let us mention that an additional motivation
stems from the sedimentation problem for rigid particles in a Stokes flow, e.g.~\cite{Batchelor-72}.
This concerns the case of particles that are heavier than the fluid, and therefore settle in the fluid.
In the corrector equation, this yields an additional force on the particles, which pumps energy into the system
and entails a crucial lack of compactness.
We refer to our very recent work~\cite{DG-20b} (see also~\cite{Gloria-19}) for a thorough discussion of the behavior of such sedimenting suspensions; although inspired by the present contribution, the analysis is much more involved and happens to require a strong mixing condition on the particle distribution even for qualitative results. Among other things, we show in~\cite{DG-20b} that the corresponding effective viscosity coincides with that for a non-sedimenting suspension, hence only depends on the geometry of the suspension.

\subsection{Main results}

Throughout, we place ourselves in dimension $d\ge2$.
We start with a suitable description of the random suspension of particles.
Let $\{x_n^\w\}_n$ denote a stationary and ergodic random point process on the ambient space~$\R^d$, constructed on a given probability space $(\Omega,\Pm)$; see Remark~\ref{rem:stat} below for a proper definition of stationarity.
Define the corresponding spherical inclusion process
\[\Ic^\w:=\bigcup_nI_n^\w,\qquad I_n^\w:=B(x_n^\w),\]
where $B(x_n^\w)$ denotes the unit ball centered at $x_n^\w$,
and assume that it satisfies the hardcore condition
\[\inf_{m\ne n}\dist(I_n^\w,I_m^\w)>\delta\quad\text{almost surely},\]
for some fixed $\delta\in(0,1)$.
Note that spherical inclusions could be replaced by random shapes under a uniform $C^2$ regularity assumption.
In addition, the deterministic lower bound on the minimal interparticle distance can be relaxed into a lower bound of the
type 
$$\E\Big[\mathds 1_{|x_n|<1}\sup_{m:m\ne n} \dist(I_n,I_m)^{-p}\Big]<\infty,$$
for some large enough power $p\ge 1$, at the price
of tracking down random constants in the proof and using Meyers-type estimates on solutions. We do however not pursue in this direction here; we believe that such conditions could be further improved, possibly in the spirit of~\cite{Jikov-90}, see also~\cite[Section~8.6]{JKO94}.

\medskip\noindent
Given a reference bounded Lipschitz domain $U$, we consider the set $\Nc_\e^\w(U)$ of all indices $n$ such that $\e(I_n^\w+\delta B)\subset U$, and we define the corresponding rescaled inclusion process $\Ic_\e^\w(U)$ in $U$,
\[\Ic_\e^\w(U):=\bigcup_{n\in\Nc_\e^\w(U)}\e I_n^\w.\]
Note that balls of this collection are at distance at least $\e\delta$ from one another and from the boundary $\partial U$. This inclusion process represents a random suspension of particles in the reference domain $U$.
We then consider these particles as suspended in a solvent described by the steady Stokes equation: the fluid velocity $u_\e^\w$ satisfies
\[-\triangle u_\e^\w+\nabla P_\e^\w=0,\qquad\Div u_\e^\w=0,\qquad\text{in $U\setminus\Ic_\e^\w(U)$},\]
and $u_\e^\w=0$ on $\partial U$. As the pressure is defined up to a constant, we choose for instance
\[\int_{U\setminus\Ic_\e^\w(U)}P_\e^\w=0.\]
Next, no-slip boundary conditions are imposed at particle boundaries; since particles are constrained to have rigid motions, this amounts to letting the velocity field~$u_\e^\w$ be extended inside particles, with the rigidity constraint
\[\D(u_\e^\w)=0,\qquad\text{in $\Ic_\e^\w(U)$},\]
where $\D(u_\e^\w)$ denotes the symmetrized gradient of $u_\e^\w$. In other words, this condition means that $u_\e^\w$ coincides with a rigid motion $V_{\e,n}^\w+\Theta_{\e,n}^\w(x-\e x_n^\w)$ inside each inclusion $\e I_n^\w$, for some $V_{\e,n}^\w\in\R^d$ and skew-symmetric matrix $\Theta_{\e,n}^\w\in\R^{d\times d}$.
Finally, assuming that the particles have the same mass density as the fluid, buoyancy forces vanish, hence
the force and torque balances on each particle take the form
\begin{gather}
\int_{\e\partial I_n^\w}\sigma(u_\e^\w,P_\e^\w)\nu=0,\label{eq:force-bal}\\
\int_{\e\partial I_n^\w}\Theta(x-\e x_n^\w)\cdot\sigma(u_\e^\w,P_\e^\w)\nu=0,\quad\text{for all $\Theta\in\Md^\Skew$},\nonumber
\end{gather}
where $\Md^\Skew\subset\R^{d\times d}$ denotes the subspace of skew-symmetric matrices,
$\sigma(u_\e^\w,P_\e^\w)$ is the usual Cauchy stress tensor,
\[\sigma(u_\e^\w,P_\e^\w)=2\D(u_\e^\w)-P_\e^\w\Id,\]
and $\nu$ stands for the outward unit normal vector at the particle boundaries.
In the physically relevant three-dimensional case $d=3$, skew-symmetric matrices $\Theta\in\Md^\Skew$ are equivalent to cross products $\theta\times$ with $\theta\in\R^3$, and equations recover their more standard form.

\medskip\noindent
In this context, modeling a dense suspension of small rigid particles in a viscous fluid with the same mass density, our homogenization result takes on the following guise.

\begin{theor1}\label{th:Stokes}
Given a bounded Lipschitz domain $U\subset\R^d$ and given a forcing $f\in\Ld^2(U)$, consider for all $\e>0$ and $\w\in\Omega$ the unique weak solution $(u_\e^\w,P_\e^\w)\in H^1_0(U)\times \Ld^2(U\setminus\Ic_\e^\w(U))$ of the Stokes problem introduced above, that is,
\begin{equation}\label{eq:Stokes}
\left\{\begin{array}{ll}
-\triangle u_\e^\w+\nabla P_\e^\w=f,&\text{in $U\setminus\Ic_\e^\w(U)$},\\
\Div u_\e^\w=0,&\text{in $U\setminus\Ic_\e^\w(U)$},\\
u_\e^\w=0,&\text{on $\partial U$},\\
\D(u_\e^\w)=0,&\text{in $\Ic_\e^\w(U)$},\\
\int_{\e\partial I_n^\w}\sigma(u_\e^\w,P_\e^\w)\nu=0,&\forall n\in\Nc_\e^\w(U),\\
\int_{\e\partial I_n^\w}\Theta(x-\e x_n^\w)\cdot \sigma(u_\e^\w,P_\e^\w)\nu=0,&\forall n\in\Nc_\e^\w(U),\,\forall\Theta\in\Md^\Skew,\\
\int_{U\setminus\Ic_\e^\w(U)} P_\e^\w=0,&
\end{array}\right.
\end{equation}
and denote by 
$\lambda:=\expec{\mathds 1_{\Ic}}$ the volume fraction of the suspension.
Then for almost all $\w$ there holds
\begin{equation*}\begin{array}{rcll}
\displaystyle u_\e^\w-\bar u&\cvf&0,\qquad&\displaystyle\text{weakly in $H^1_0(U)$,}\\
\vspace{-0.3cm}&&&\\
\displaystyle (P_\e^\w-\bar P-\bb:\D(\bar u))\mathds1_{U\setminus\Ic_\e^\w(U)}&\cvf&0,\qquad&\displaystyle\text{weakly in $\Ld^2(U)$,}
\end{array}\end{equation*}
where $(\bar u,\bar P)\in H^1_0(U)\times\Ld^2(U)$ is the unique weak solution of the homogenized Stokes flow
\begin{equation}\label{eq:Stokes-hom}
\left\{\begin{array}{ll}
-\Div2\Bb \D(\bar u)+\nabla\bar P=(1-\lambda) f,&\text{in $U$},\\
\Div\bar u=0,&\text{in $U$},\\
\bar u=0,&\text{on $\partial U$},\\
\int_U\bar P=0,&
\end{array}\right.
\end{equation}
and the effective constants are as follows:
\begin{enumerate}[\quad$\bullet$]
\item the effective diffusion tensor $\Bb$ is a positive definite symmetric linear map on symmetric trace-free matrices $\Md^\Sym_0\subset\R^{d\times d}$, and is defined for all $E\in\Md_0^\Sym$ by
\begin{equation}\label{eq:def-B}
E:\Bb E\,:=\,\expec{|\D(\psi_{E})+E|^2};
\end{equation}
\item $\bb$ is a symmetric trace-free matrix and is given for all $E\in\Md_0^\Sym$ by
\begin{equation}\label{eq:def-b}
\bb:E\,:=\,\frac1d\,\E\bigg[{\sum_n\frac{\mathds1_{I_n}}{|I_n|}\int_{\partial I_n}(x-x_n)\cdot\sigma(\psi_E+Ex,\Sigma_E)\nu}\bigg];
\end{equation}
\end{enumerate}
where $\nabla\psi_E\in\Ld^2(\Omega;\Ld^2_\loc(\R^d)^{d\times d})$ is the unique stationary gradient solution with vanishing expectation and $\Sigma_E\in\Ld^2(\Omega;\Ld^2_\loc(\R^d\setminus\Ic))$ is the unique associated stationary pressure with vanishing expectation for the following infinite-volume corrector problem, cf.~Proposition~\ref{prop:corr-Stokes}:
for almost all $\w$,
\begin{equation}\label{eq:corr-Stokes}
\left\{\begin{array}{ll}
-\triangle\psi_E^\w+\nabla \Sigma_E^\w=0,&\text{in $\R^d\setminus\Ic^\w$},\\
\Div\psi_E^\w=0,&\text{in $\R^d\setminus\Ic^\w$},\\
\D(\psi_E^\w+Ex)=0,&\text{in $\Ic^\w$},\\
\int_{\partial I_n^\w}\sigma(\psi_E^\w+Ex,\Sigma_E^\w)\nu=0,&\forall n,\\
\int_{\partial I_n^\w}\Theta(x-\e x_n^\w)\cdot\sigma(\psi_E^\w+Ex,\Sigma_E^\w)\nu=0,&\forall n,\,\forall\Theta\in\Md^\Skew.
\end{array}\right.
\end{equation}
Moreover, provided $f\in\Ld^p(U)$ for some $p>d$, for almost all~$\w$, we have a corrector result
for the velocity field,
\[\Big\|u_\e^\w-\bar u-\e\sum_{E\in\Ec}\psi_E^\w(\tfrac\cdot\e)\nabla_E\bar u\Big\|_{H^1(U)}\to0,\]
and for the pressure field,
\[\inf_{\kappa\in\R}~\Big\|P_\e^\w-\bar P-\bb:\D(\bar u)-\sum_{E\in\Ec}(\Sigma_E^\w\mathds1_{\R^d\setminus\Ic^\w})(\tfrac\cdot\e)\nabla_E\bar u-\kappa\Big\|_{\Ld^2(U\setminus\Ic_\e^\w(U))}\to0,\]
where the sums run over an orthonormal basis $\Ec$ of $\Md_0^\Sym$.
\end{theor1}

\begin{rem}[Buoyancy and sedimentation problem]\label{rem:sediment}
If particles do not have the same mass density as the solvent fluid, a nontrivial buoyancy must be taken into account in the force balance~\eqref{eq:force-bal}: denoting by $g\in C_b(U)^d$ the buoyancy, this equation is replaced by
\begin{equation}\label{eq:buoyancy}
\frac1\e \int_{\e I_n^\w}g+\int_{\e\partial I_n^\w}\sigma(u_\e^\w,P_\e^\w)\nu=0.
\end{equation}
The scaling in $\e$ is such that surface and volumetric forces have the same order uniformly in $\e$ (it is equivalent, in sedimentation experiments, to increasing the size of the tank, rather than decreasing the size of the particles).
Since an a priori diverging amount $O(\frac1\e)$ of energy is then pumped into the system, it needs to be compensated by modifying the definition of correctors~\eqref{eq:corr-Stokes}. This is fully analyzed in our companion article~\cite{DG-20a} under strong mixing conditions, where we show in particular that the effective viscosity is not affected by the settling process.
A weak sedimentation regime can however be considered as a direct adaptation of our present analysis, replacing~\eqref{eq:buoyancy} by  
\begin{equation*}
\int_{\e  I_n^\w}g+\int_{\e\partial I_n^\w}\sigma(u_\e^\w,P_\e^\w)\nu=0,
\end{equation*}
in which case the buoyancy vanishes in the limit (as the quotient of a volumetric over a surfacic term in the limit of small particles), and 
the effective equation is then obtained by adding a forcing term to~\eqref{eq:Stokes-hom}  in form of
\begin{equation*}
\left\{\begin{array}{ll}
-\Div2\Bb\D(\bar u)+\nabla\bar P=(1-\lambda)f+\lambda g,&\text{in $U$},\\
\Div\bar u=0,&\text{in $U$},\\
\bar u=0,&\text{on $\partial U$},\\
\int_U\bar P=0.&
\end{array}\right.
\end{equation*}
This simpler problem is however strictly distinct from the proper sedimentation regime.
\end{rem}

\begin{rem}[Stationary setting]\label{rem:stat}
We briefly recall  the standard formulation of the stationary setting, make precise probabilistic assumptions, and recall some useful notation and constructions for stationary random fields.
\begin{enumerate}[(i)]
\item \emph{Stationarity and probabilistic assumptions.} As is customary in stochastic homogenization theory, e.g.~\cite[Section~7]{JKO94}, stationarity is most conveniently defined via a measurable action  $\{\tau_x\}_{x\in\R^d}$ of the translation group $(\R^d,+)$ on the underlying probability space $(\Omega,\Pm)$. More precisely, the space is endowed with measurable maps $\tau_x:\Omega\to\Omega$ that satisfy
\begin{itemize}
\item $\tau_x\circ\tau_y=\tau_{x+y}$ for all $x,y\in\R^d$;
\item $\pr{\tau_x A}=\pr{A}$ for all $x\in\R^d$ and measurable $A\subset\Omega$;
\item the map $\R^d\times\Omega\to\Omega:(x,\w)\mapsto\tau_x\w$ is jointly measurable;
\end{itemize}
and this action is assumed to be ergodic in the sense that any random variable $\tilde\phi\in\Ld^1(\Omega)$ that is $\tau$-invariant (i.e., $\tilde\phi(\tau_x\cdot)=\tilde\phi$ almost surely for all $x$) is almost surely constant.
The point process $\{x_n^\w\}_n$ is then said to be stationary (with respect to $\tau$) if $\{x_n^{\tau_{x}\omega}\}_n=\{x+x_n^\omega\}_n$ for all $x,\w$.
\smallskip\item \emph{Stationary extensions.}
A function $\phi:\R^d\times\Omega\to\R$ is said to be stationary if there exists a measurable map $\tilde\phi:\Omega\to\R$ such that $\phi(x,\w)=\tilde\phi(\tau_{-x}\w)$ for all $x,\w$.
The joint measurability assumption for the action then ensures that $\phi$ is jointly measurable, which in view of a result by von Neumann is equivalent to stochastic continuity, that is, $\pr{|\phi(x+y,\cdot)-\phi(x,\cdot)|>\delta}\to0$ as $y\to0$ for all $x$ and $\delta>0$, cf.~\cite[Section~7]{JKO94}.
Stationarity then yields a bijection between random variables $\tilde\phi:\Omega\to\R$ and stationary measurable functions $\phi:\R^d\times\Omega\to\R$. The function $\phi$ is referred to as the stationary extension of the random variable~$\tilde\phi$.
The subspace of stationary functions $\phi\in\Ld^2(\Omega;\Ld^2_\loc(\R^d))$ is then identified with the Hilbert space $\Ld^2(\Omega)$, and the (spatial) weak gradient $\nabla$ on locally square integrable functions turns into a linear operator on $\Ld^2(\Omega)$. We also
define $H^s(\Omega)$ as the subspace of random variables $\tilde\phi\in\Ld^2(\Omega)$ with stationary extension $\phi\in\Ld^2(\Omega;H^s_\loc(\R^d))$. We often use the short-hand notation $\phi^\w(x):=\phi(x,\w)$.
\qedhere
\end{enumerate}
\end{rem}

\subsection*{Notation}
\begin{itemize}
\item For vector fields $u,u'$ and matrix fields $T,T'$, we set $(\nabla u)_{ij}=\nabla_ju_i$, $(\Div T)_i=\nabla_jT_{ij}$, $T:T'=T_{ij}T'_{ij}$, $(u\otimes u')_{ij}=u_iu'_j$, $(T^s)_{ij}=\frac12(T_{ij}+T_{ji})$, $\D(u)=(\nabla u)^s$. For a vector field $u$ and a matrix $E$, we also write $\nabla_Eu=E:\nabla u$. We systematically use {Einstein's summation convention} on repeated indices.
\smallskip\item We denote by $\Md=\R^{d\times d}$ the space of $d\times d$ matrices, by $\Md_0^\Sym$  the subset of symmetric trace-free matrices, and by $\Md^\Skew$  the subset of skew-symmetric matrices.
\smallskip\item We denote by $C\ge1$ any constant that only depends on the dimension $d$,  on the reference domain $U$, and on the hardcore constant $\delta\in (0,1)$. We use the notation~$\lesssim$ (resp.~$\gtrsim$) for $\le C\times$ (resp.\@ $\ge\frac1C\times$) up to such a multiplicative constant $C$. We add subscripts to $C,\lesssim,\gtrsim$ in order to indicate dependence on other parameters.
\smallskip\item The ball centered at $x$ of radius $r$ in $\R^d$ is denoted by $B_r(x)$, and we simply write $B(x)=B_1(x)$, $B_r:=B_r(0)$, and $B=B_1(0)$.
\end{itemize}

\section{Construction of correctors}\label{sec:corr}
This section is devoted to the construction of a suitable solution to the Stokes corrector equation~\eqref{eq:corr-Stokes}.

\begin{prop}\label{prop:corr-Stokes}
Under the assumptions and notation of Theorem~\ref{th:Stokes}, for all $E\in\Md_0^\Sym$, there exist a unique random field $\psi_E\in \Ld^2(\Omega;H^1_\loc(\R^d)^d)$ and a unique pressure field $\Sigma_E\in  \Ld^2(\Omega;\Ld^2_\loc(\R^d\setminus\Ic))$ such that
\begin{enumerate}[(i)]
\item For almost all $\omega$ the realizations $\psi_E^\omega\in H^1_\loc(\R^d)$ and $\Sigma_E^\omega\in\Ld^2_\loc(\R^d\setminus\Ic^\w)$ satisfy
\begin{equation}\label{e.corr-eq}
\quad\left\{\begin{array}{ll}
-\triangle\psi_E^\omega+\nabla \Sigma_E^\omega=0,&\text{in $\R^d\setminus\Ic^\w$},\\
\Div\psi_E^\w=0,&\text{in $\R^d\setminus\Ic^\w$},\\
\D(\psi_E^\omega+Ex)=0,&\text{in $\Ic^\omega$},\\
\fint_{\partial I_n^\omega}\sigma(\psi_E^\omega+Ex,\Sigma_E^\omega)\nu=0,&\forall n,\\
\fint_{\partial I_n^\omega}\Theta(x-\e x_n^\w)\cdot\sigma(\psi_E^\omega+Ex,\Sigma_E^\omega)\nu=0,&\forall n,\,\forall\Theta\in\Md^\Skew.
\end{array}\right.
\end{equation}
\item The corrector gradient $\nabla\psi_E$ and the pressure $\Sigma_E\mathds1_{\R^d\setminus \Ic}$ are stationary\footnote{That is, $\nabla\psi_E^\w(x+y)=\nabla\psi_E^{\tau_{-y}\omega}(x)$ 
and $\Sigma_E^\w(x+y)\mathds1_{\R^d\setminus \Ic^\w}(x+y)=\Sigma_E^{\tau_{-y}\w}(x)\mathds1_{\R^d\setminus \Ic^{\tau_{-y}\w}}(x)$ for all $x,y,\omega$, cf.~Remark~\ref{rem:stat}.},
with
\begin{gather*}
\expecm{\nabla\psi_E}=0,\quad\expecm{\Sigma_E\mathds1_{\R^d \setminus \Ic}}=0,\\
\expecm{|\nabla\psi_E|^2}+\expecm{\Sigma_E^2\mathds{1}_{\R^d \setminus \Ic}}\,\lesssim\,|E|^2,
\end{gather*}
and we choose the anchoring $\int_{B}\psi_E^\omega=0$ for the corrector.
\end{enumerate}
In addition, the following properties hold:
\begin{enumerate}[(i)]\setcounter{enumi}{2}
\item \emph{Ergodic theorem for averages of corrector gradient and pressure:} for almost all $\omega$,
\begin{equation*}\begin{array}{rcll}
(\nabla\psi_E^\omega)(\tfrac\cdot\e)&\rightharpoonup&\expec{\nabla\psi_E}=0 \quad &\text{weakly in } \Ld^2_\loc(\R^d) \text{ as }\e\downarrow0,
\\\vspace{-0.3cm}\\
(\Sigma_E^\omega\mathds1_{\R^d \setminus \Ic^\w})(\tfrac\cdot\e)&\rightharpoonup&\expecm{\Sigma_E\mathds1_{\R^d \setminus \Ic}}=0  \quad&\text{weakly in } \Ld^2_\loc(\R^d) \text{ as }\e\downarrow0.
\end{array}
\end{equation*}
\item \emph{Sublinearity of the corrector:} for almost all $\omega$, for all $q<\frac{2d}{d-2}$,
\[\e\psi_E^\omega(\tfrac\cdot\e)\,\to\,0\quad \text{strongly in }\Ld^q_\loc(\R^d) \text{ as }\e\downarrow0
.\qedhere\]
\end{enumerate}
\end{prop}

\begin{proof}
We start by defining suitable functional subspaces of $\Ld^2(\Omega)^{d\times d}$ that are tailored for the study of the corrector equation~\eqref{e.corr-eq}.
We first consider the subspace of potential fields with vanishing trace, 
\[\Lc^2(\Omega)\,:=\,\big\{\tilde\Psi\in\Ld^2(\Omega)^{d\times d}:\expecm{\tilde\Psi}=0,~\Tr\tilde\Psi=0,~\expecm{\tilde\Psi:(\nabla\times\tilde\chi)}=0~~\forall\tilde\chi\in H^1(\Omega)^d\big\}.\]
Using stationary extensions, cf.~Remark~\ref{rem:stat}, it is well-known (e.g.~\cite[Section~7]{JKO94}) that this space is equivalently given by
\begin{multline*}
\Lc^2(\Omega)\,=\,\big\{\tilde\Psi\in\Ld^2(\Omega)^{d\times d}:\expecm{\tilde\Psi}=0,\,\text{and}\,~\exists\psi\in\Ld^2(\Omega;\Ld^2_\loc(\R^d)^d)\\
\,\text{with}\,\Psi=\nabla\psi~\text{and}\,\Div\psi=0\big\},
\end{multline*}
where the differential constraints are more clearly interpreted.
We further incorporate the specific boundary conditions of the corrector equation~\eqref{e.corr-eq} into the functional space, defining for $E\in\Md_0$ the convex set
\begin{multline*}
\Lc^2_E(\Omega)\,:=\,\big\{\tilde\Psi\in\Lc^2(\Omega)\,:\,\exists\psi\in\Ld^2(\Omega;\Ld^2_\loc(\R^d)^d)\,\text{with}\,\Psi=\nabla\psi,\\
\,\text{and with}\,
\D(\psi^\w+Ex)=0~\text{in $\Ic^\w$}~\forall\w\big\}.
\end{multline*}
As we shall check in Substep~3.1 below, $\Lc^2_E(\Omega)$ is not empty.
Differences of elements of~$\Lc^2_E(\Omega)$ belong to the vector space
\begin{multline*}
\Lc^2_0(\Omega)\,:=\,\big\{\tilde\Psi\in\Lc^2(\Omega)\,:\,\exists\psi\in\Ld^2(\Omega;\Ld^2_\loc(\R^d)^d)\,\text{with}\,\Psi=\nabla\psi,\\
~\text{and with}\,
\D(\psi^\w)=0~\text{in $\Ic^\w$}~\forall\w\big\}.
\end{multline*}
A well-known density result (e.g.~\cite[Section~7]{JKO94}) ensures that
\[\Lc^2(\Omega)\,=\,\adh_{\Ld^2(\Omega)^{d\times d}}\big\{\nabla\tilde\psi\,:\,\tilde\psi\in H^1(\Omega)^{d},\,\Div\tilde\psi=0\big\}.\]
Likewise,
\begin{equation}\label{eq:density-L20}
\Lc^2_0(\Omega)=\adh_{\Ld^2(\Omega)^{d\times d}} \Kc^2_0(\Omega),
\end{equation}
with
\begin{equation*}
\Kc^2_0(\Omega)\,:=\,\big\{\nabla\tilde\psi\,:\,\tilde\psi\in H^1(\Omega)^{d},\,\Div\tilde\psi=0,~\text{and}~
\D(\psi^\w)=0~\text{in $\Ic^\w$}~\forall\w\big\}.
\end{equation*}
Once these spaces are introduced, the structure of the proof is as follows.
We first show that for a solution $(\psi_E,\Sigma_E)$ of~(i)--(ii) the gradient $\nabla\psi_E$ is the unique Lax-Milgram solution in $\Lc^2_E(\Omega)$ of an abstract coercive problem on the probability space.
We then argue that conversely this unique solution indeed provides a solution of~(i)--(ii) in a weak sense in the physical space. Finally, from such a weak formulation, we reconstruct the pressure and establish the desired estimates (iii)--(iv).
The proof is split into five main steps.

\medskip
\step1 From  (i)--(ii) to an abstract problem in $\Lc^2_E(\Omega)$.\\
Let $\psi_E$ be a solution of (i)--(ii).
In particular, $\Psi_E:=\nabla \psi_E$ is stationary and defines an element~$\tilde\Psi_E\in \Lc^2_E(\Omega)$.
We claim that it satisfies
\begin{equation}\label{e.weak-form-corr}
\expecm{\tilde\Phi :\tilde\Psi_E} =0,\qquad\text{for all $\tilde\Phi\in\Lc^2_0(\Omega)$}.
\end{equation}
By density~\eqref{eq:density-L20}, it is enough to prove \eqref{e.weak-form-corr} for all $\tilde\Phi \in \Kc^2_{0}(\Omega)$.
Let $\tilde\Phi \in \Kc^2_{0}(\Omega)$ be given by $\tilde\Phi=\nabla\tilde\phi$ for some $\tilde\phi \in H^1(\Omega)^d$ with $\Div\tilde\phi=0$ and with $\D(\phi^\w)=0$ in $\Ic^\w$ for all $\w$.
In view of the hardcore condition, for all $R>0$, we can construct a cut-off function $\eta_R^\w$ supported in $B_{R+3}$ with $\eta_R^\w=1$ on $B_R$ and with $|\nabla \eta_R^\w|\lesssim_\delta 1$, such that $\eta_R^\w$ is constant in $I_n^\w+\frac\delta4B$ for all $n$.
Since $\psi_E^\w$ is divergence-free, an integration by parts yields
\[\int_{\R^d}\nabla (\eta_R^\w \phi^\w):\nabla \psi_E^\w=
2\int_{\R^d}\nabla(\eta_R^\w \phi^\w):\D(\psi_E^\w),\]
and thus, since $\int_{\R^d}\nabla(\eta_R^\w\phi^\w)=0$ and since $\D(\psi_E^\w)+E=0$ in $\Ic^\w$,
\[\int_{\R^d}\nabla (\eta_R^\w \phi^\w):\nabla \psi_E^\w=
2\int_{\R^d\setminus\Ic^\w}\nabla(\eta_R^\w \phi^\w):(\D(\psi_E^\w)+E).\]
Recalling the definition $\sigma(\psi_E^\w+Ex,\Sigma_E^\w)=2(\D(\psi_E^\w)+E)-\Sigma_E^\w\Id$,
integrating by parts, and using the corrector equation~\eqref{e.corr-eq}, we obtain
\begin{multline}
\int_{\R^d}\nabla (\eta_R^\w \phi^\w):\nabla \psi_E^\w- \int_{\R^d \setminus \Ic^\w} \Div(\eta_R^\w \phi^\w)\,\Sigma_E^\w~=~\int_{\R^d\setminus\Ic^\w}\nabla(\eta_R^\w \phi^\w):\sigma(\psi_E^\w+Ex,\Sigma_E^\w)\\
~=~-\sum_n\int_{\partial I_n^\w}\eta_R^\w \phi^\w\cdot\sigma(\psi_E^\w+Ex,\Sigma_E^\w)\nu.\label{eq:weak-form-eqn00+}
\end{multline}
For all $n$, since $\eta_R^\w$ is constant in $I_n^\w$ and since $\phi^\w$ takes the special form $\kappa^\w_n+\Theta^\w_n(x-x_n^\w)$ in $I_n^\w$ for some $\kappa_n^\w\in\R^d$ and $\Theta_n^\w\in\Md^\Skew$, the boundary condition in~\eqref{e.corr-eq} for $\psi_E^\w$ on $\partial I_n^\w$ precisely yields
\begin{equation*}
\int_{\partial I_n^\w} \eta_R^\w\,\phi^\w \cdot\sigma(\psi_E^\w+Ex,\Sigma_E^\w)\nu
\,=\,0.
\end{equation*}
The weak form~\eqref{eq:weak-form-eqn00+} of the equation thus becomes, after expanding the gradients and recalling that $\Div\phi^\w=0$,
\begin{equation*}
\int_{\R^d} \eta_R^\w \nabla \phi^\w:\nabla \psi_E^\w
\,=\,- \int_{\R^d} \phi^\w\otimes\nabla \eta_R^\w : \nabla \psi_E^\w+ \int_{\R^d \setminus \Ic^\w}\Sigma_E^\w\nabla \eta_R^\w\cdot \phi^\w,
\end{equation*}
which by the properties of $\eta_R^\w$ we rewrite as
\begin{multline*}
\int_{B_R}   \nabla \phi^\w: \nabla \psi_E^\w\,=\,-\int_{B_{R+3}\setminus B_R} \eta_R^\w \nabla \phi^\w: \nabla \psi_E^\w
\\
 -\int_{B_{R+3}\setminus B_R} \phi^\w\otimes\nabla \eta_R^\w : \nabla \psi_E^\w
 +\int_{B_{R+3}\setminus B_R }\Sigma_E^\w\mathds{1}_{\R^d \setminus \Ic^\w}\nabla \eta_R^\w\cdot  \phi^\w.
\end{multline*}
Taking the expectation, using the stationarity of $\phi$, $\nabla \psi_E$, and $\Sigma_E$, as well as the a priori bounds~(ii) and the boundedness of $\eta_R$, we obtain from Cauchy-Schwarz' inequality for all~$R\ge 1$,
$$
\big|\expecm{\nabla\tilde\phi : \tilde\Psi_E }\big|\,\lesssim\, \frac1R |E| \, \expec{|\phi|^2+|\nabla \phi|^2}^\frac12,
$$
and the claim follows from the arbitrariness of $R$.

\medskip
\step2 Well-posedness of the abstract problem~\eqref{e.weak-form-corr}.\\
In this step, we argue that there exists a unique solution $\tilde\Psi_E \in \Lc^2_E(\Omega)$ to the problem~\eqref{e.weak-form-corr}.
As we shall check in Substep~3.1 below, the convex set $\Lc^2_E(\Omega)$ is not empty, so that we may choose a reference field $\tilde\Psi^0_E\in \Lc^2_E(\Omega)$. Writing $\tilde\Psi_E=\tilde\Psi^0_E+\tilde\Psi^1_E$
for some $\tilde\Psi^1_E\in \Lc^2_0(\Omega)$, the equation~\eqref{e.weak-form-corr} for $\tilde\Psi_E$ is equivalent to the following equation for $\tilde\Psi^1_E$,
\begin{equation}\label{e.weak-form-probab}
\expecm{\tilde\Phi:\tilde\Psi^1_E}=-\expecm{\tilde\Phi:\tilde\Psi^0_E}\qquad\text{for all $\tilde\Phi\in\Lc^2_0(\Omega)$}.
\end{equation}
The existence and uniqueness of the solution $\tilde \Psi^1_E$ to this equation then follow from the Lax-Milgram theorem in the Hilbert space $\Lc^2_0(\Omega)$.

\medskip

\step3 From the abstract problem~\eqref{e.weak-form-corr} to a weak formulation of~(i).\\
Let $\tilde\Psi_E \in \Lc^2_E$ denote the unique solution of~\eqref{e.weak-form-corr} as constructed in Step~2, 
which can be written as $\Psi_E=\nabla \psi_E$ in terms of the almost surely unique random field $\psi_E\in\Ld^2(\Omega;H^1_\loc(\R^d)^d)$ that satisfies
the anchoring condition $\int_B \psi_E=0$ at the origin.
By construction, $\Div\psi_E^\w=0$, and $\D(\psi_E^\w+Ex)=0$ in $\Ic^\w$ for all $\w$.
Next, we prove that 
$\psi_E$ satisfies the following weak formulation of~\eqref{e.corr-eq}: for almost all $\w$,  
\begin{equation}\label{e:back-phys}
\int_{\R^d} \nabla \phi : \nabla \psi_E^\w\,=\,0,
\end{equation}
for all test functions $\phi$ in the class
\begin{equation*}
\mathcal C^\w:=\big\{\phi\in H^1(\R^d)^d\,:\,\text{$\phi$ has compact support, }\Div\phi=0,~
\text{and }\D(\phi)=0\text{ in $\Ic^\w$}\big\}.
\end{equation*}
We split the proof of~\eqref{e:back-phys} into two further substeps.

\medskip
\substep{3.1} Definition of a suitable map $\Mk^\w:H^1_{c,\Div}(\R^d)^d\to\Cc^\w$, where $H^1_{c,\Div}(\R^d)^d$ stands for the subspace $\{\zeta\in H^1(\R^d)^d:\zeta\text{ has compact support and }\Div\zeta=0\}$
of $H^1(\R^d)^d$.

\medskip\noindent
Choose a map $\Mk_\circ:H^1_{\Div}(B_{1+\delta/2})^d\to H^1_{\Div}(B_{1+\delta/2})^d$ that satisfies for all $\zeta\in H^1_{\Div}(B_{1+\delta/2})^d$:
\begin{enumerate}[(1)]
\item $\Mk_\circ \zeta-\zeta\in H^1_0(B_{1+\delta/2})$;
\item $\D(\Mk_\circ \zeta)=0$ in $B$;
\item if $\D(\zeta)=0$ in $B$, then $\Mk_\circ \zeta=\zeta$;
\item $\|\nabla \Mk_\circ \zeta\|_{\Ld^2(B_{1+\delta/2})} \,\lesssim\, \|\nabla\zeta\|_{\Ld^2(B_{1+\delta/2})}$.
\end{enumerate}
Such a map $\Mk_\circ$ can for instance be constructed as follows,
\begin{multline}\label{e.defM}
\Mk_\circ \zeta  \,:=\, \mathrm{arginf}\,\big\{\|\nabla \xi-\nabla \zeta\|_{\Ld^2(B_{1+\delta/2})}^2\,:\,
\xi \in \zeta+ H^1_0(B_{1+\delta/2}),\,\Div\xi=0,\\
\text{and }\D(\xi)=0\text{ in $B$}\big\}.
\end{multline}
Since this is the minimization of a strictly convex lower-semicontinuous functional on a convex set, the infimum is attained and unique provided the convex set is nonempty.
Choosing $\kappa=\fint_{B_{1+\delta/2}\setminus B}\zeta$ and $\Theta=0$, it suffices to check that there exists $\xi\in\zeta+ H^1_0(B_{1+\delta/2})$ with
\begin{equation}\label{eq:real-pr}
\Div\xi=0\qquad\text{and}\qquad\xi|_B=\kappa.
\end{equation}
For that purpose, choose $u_\zeta \in H^1_0(B_{1+\delta/2})$ that coincides with $-\zeta+\kappa$ on~$B$.
In view of the compatibility condition
$$
\int_{B_{1+\delta /2}\setminus B} \Div u_\zeta=\int_{\partial B_{1+\delta /2}} u_\zeta \cdot \nu-\int_{\partial B} u_\zeta \cdot \nu = \int_{\partial B} \zeta \cdot \nu = \int_B \Div \zeta=0, 
$$
a standard use of the Bogovskii operator in form of~\cite[Theorem~III.3.1]{Galdi} ensures that this $u_\zeta$ can be modified in $B_{1+\delta/2}\setminus B$ (without changing its boundary values) to be divergence-free in $B_{1+\delta/2}\setminus B$ (hence in the whole of $B_{1+\delta/2}$), with the estimate 
\[\|\nabla u_\zeta\|_{\Ld^2(B_{1+\delta / 2}\setminus B)} \,\lesssim\, \|\zeta-\kappa \|_{H^\frac12(\partial B)}.\]
In particular, by a trace estimate and Poincar\'e's inequality, this yields
\begin{equation*}
\|\nabla u_\zeta\|_{\Ld^2(B_{1+\delta / 2}\setminus B)}\,\lesssim \,  \|\zeta-\kappa \|_{H^1(B_{1+\delta/2}\setminus B)}\,\lesssim_\delta\,   \|\nabla \zeta\|_{\Ld^2(B_{1+\delta/2}\setminus B)}.
\end{equation*}
The function $\xi_\zeta:=\zeta+ u_\zeta\in\zeta+ H^1_0(B_{1+\delta/2})$ then satisfies~\eqref{eq:real-pr} and
\[\|\nabla\xi_\zeta\|_{\Ld^2(B_{1+\delta/2})}\lesssim_\delta\|\nabla\zeta\|_{\Ld^2(B_{1+\delta/2})}.\]
This implies that $\Mk_\circ$ in~\eqref{e.defM} is well-defined and indeed satisfies the properties~(1)--(4). 

\medskip\noindent
Next, for $\zeta\in H^1_{c,\Div}(\R^d)^d$, we extend $\Mk_\circ\zeta$ by $\zeta$ outside $B_{1+\delta/2}$, and for all $x\in \R^d$ we denote by $\Mk_x$ the corresponding operator when the origin $0$ is replaced by $x$.
For all $\w$, we then define the operator $\Mk^\w:=\prod_n\Mk_{x_n^\w}$, which indeed maps $H^1_{c,\Div}(\R^d)^d$ to $\calC^\w$ as desired.

\medskip\noindent
We conclude this construction of $\Mk^\w$ with a weak continuity result: for all bounded domains~$D$ and all sequences $(\zeta_n)_n$ of divergence-free functions compactly supported in $D$, if  $\zeta_n \cvf \zeta$ weakly in $H^1(D)$, then for all $\w$ we have $\Mk^\w \zeta_n \rightharpoonup \Mk^\w \zeta$ in $H^1(D)$.
In view of the above construction of $\Mk^\w$, it is enough to prove this continuity result at the level of the elementary map~$\Mk_\circ$.
Since the sequence $(\Mk_\circ \zeta_n)_n$ is bounded in $H^1(B_{1+\delta/2})$, it converges to some $\xi$ along 
a subsequence (not relabelled), which is necessarily an admissible test function for the minimization problem \eqref{e.defM} for $\Mk_\circ\zeta$.
It remains to argue that it coincides with the desired minimizer $\Mk_\circ\zeta$.
To this aim, we use that the unique minimizers $\Mk_\circ \zeta_n$ of \eqref{e.defM} are characterized by the following
Euler-Lagrange equations: for all  $\xi' \in H^1_0(B_{1+\delta/2})$ with $\Div \xi'=0$ and with
$\D(\xi')=0$ in $B$,
$$
\int_{B_{1+\delta /2}} (\nabla \Mk_\circ \zeta_n-\nabla \zeta_n) : \nabla \xi'=0,
$$
in which we may pass to the limit in $n$ in form of 
$$
\int_{B_{1+\delta/2}} (\nabla \xi-\nabla \zeta) : \nabla \xi'=0,
$$
thus recovering the Euler-Lagrange equation for  $\Mk_\circ\zeta$. This entails $\xi=\Mk_\circ \zeta$ and ensures the convergence of the whole sequence.

\medskip\noindent
We now quickly argue that a similar argument ensures that the convex set $\Lc^2_E(\Omega)$ is not empty.
Choose $u_{E} \in H^1_0(B_{1+\delta/2})$ that coincides with $x\mapsto-Ex$ in $B$.
In view of the compatibility condition
$$
\int_{B_{1+\delta/2}\setminus B} \Div u_{E}=\int_{\partial B_{1+\delta/2}} u_{E} \cdot \nu-\int_{\partial B} u_{E} \cdot \nu = \int_{\partial B} Ex \cdot \nu = |B|\Tr E=0, 
$$
a standard use of the Bogovskii operator in form of~\cite[Theorem~III.3.1]{Galdi} ensures that $u_{E}$ can be chosen divergence-free in $B_{1+\delta/2}\setminus B$ (hence in the whole of $B_{1+\delta/2}$), with the estimate 
$$
\|\nabla u_{E}\|_{\Ld^2(B_{1+\delta/2}\setminus B)}  \,\lesssim\,|E|.
$$
We may then define the stationary function $\phi=\sum_n u_{E}(\cdot-x_n)$, which 
is such that $\tilde \Phi=\nabla \tilde \phi$ belongs to $\Lc^2_E(\Omega)$ by construction.

\medskip

\substep{3.2} Proof of \eqref{e:back-phys}.\\
Given a vector field $\phi \in H^1_{c,\Div}(\R^d)^d$ and given a random variable $\tilde\chi \in \Ld^2(\Omega)$, we define $\Phi$ as the stationarization of the product $\tilde\chi\nabla \Mk\phi$, that is,
$$
\Phi(x, \omega)\,:=\, \int_{\R^d}\tilde\chi(\tau_y\omega)\,\nabla(\Mk^{\tau_y \omega} \phi)(x+y)\,dy,
$$
which is well-defined in $\Ld^2(\Omega,\Ld^2_\loc(\R^d)^{d\times d})$ since $\phi$ (hence $\sup_\w|\Mk^\w\phi|$) is compactly supported.
On the one hand, $\Phi$ is obviously a stationary random field: for all $x,z,\w$,
\begin{eqnarray*}
\Phi(x+z, \omega)&=& \int_{\R^d}\tilde\chi(\tau_{y} \omega) \nabla \Mk^{\tau_y \omega} \phi(x+z+y)\,dy
\\
&=&\int_{\R^d} \tilde\chi(\tau_{y-z} \omega) \nabla \Mk^{\tau_{y-z} \omega} \phi(x+y)dy
\\
&=&\Phi(x, \tau_{-z}\omega).
\end{eqnarray*}
On the other hand, the definition of $\Mk$ ensures that $\tilde\Phi$ belongs to $\Lc_0^2$, which makes it an admissible test function for~\eqref{e.weak-form-corr}. By stationarity of $\Psi_E=\nabla\psi_E$ and of $\Ic$
in the form $(\Psi_E^\omega{\mathds 1}_{\R^d \setminus \Ic^\w})(0)=(\Psi_E^{\tau_y \omega}{\mathds 1}_{\R^d \setminus \Ic^{\tau_y \omega}})(y)$, and since the group action preserves the probability measure, we find
\begin{eqnarray*}
0&=&\expecm{\tilde\Phi:\tilde\Psi_E}
\\
&=&\int_\Omega\bigg(\int_{\R^d}\tilde\chi(\tau_{y} \omega) \nabla(\Mk^{\tau_y \omega}\phi)(y)\,dy :\tilde\Psi_E^\omega\bigg)d\Pm(\omega)
\\
&=&\int_\Omega\bigg({\int_{\R^d}\tilde\chi(\tau_{y} \omega) \nabla(\Mk^{\tau_y \omega}\phi)(y) :\nabla\psi_E^{\tau_y \omega}(y)\,dy}\bigg)d\Pm(\omega)
\\
&=&\expec{\tilde\chi \int_{\R^d}\nabla(\Mk\phi):\nabla \psi_E}.
\end{eqnarray*}
By the arbitrariness of $\tilde\chi$, this implies that for any compactly supported vector field $\phi \in H^1_{\Div}(\R^d)^d$ there holds for almost all $\omega$,
$$
\int_{\R^d}\nabla(\Mk^{\omega} \phi):\nabla \psi_E^{\omega}\,=\,0.
$$
By a density argument together with the weak continuity of $\Mk^\w$ as established in Substep~3.1, we deduce that for almost all $\w$ this actually holds for all compactly supported vector fields $\phi \in H^1_{\Div}(\R^d)^d$.
Given $\w$, for $\phi^\w$ in the (realization-dependent) class $\Cc^\w$, there holds $\Mk^\omega \phi^\w =\phi^\w$ and the conclusion~\eqref{e:back-phys} follows.

\medskip
\step4 Reconstruction of the pressure.\\
In Step~3, we proved that the unique solution $\Psi_E=\nabla \psi_E$ of the abstract problem~\eqref{e.weak-form-corr}
also satisfies the weak formulation~\eqref{e:back-phys} of the corrector equation~\eqref{e.corr-eq}. In addition, note that the construction of Step~3 yields the bound $\expec{|\nabla\psi_E|^2}\lesssim|E|^2$.
In the present step, we show that one can construct a stationary pressure field $\Sigma_E$ such that for almost all $\w$ the vector field $\psi_E^\w$ is a classical solution of the corrector equation~\eqref{e.corr-eq},
and that $\Sigma_E$ and $\nabla \psi_E$ satisfy (ii).
We split the proof into five further substeps.

\medskip 
\substep{4.1} Reconstruction of a pressure field $\bar \Sigma_E$.\\
For $R\ge2$, consider the bounded Lipschitz domain
\[D_R^\w:=B_R\cup\bigcup_{n:I_n^\w\cap B_R\ne\varnothing}(I_n^\w+\tfrac\delta 2B).\]
In view of~\eqref{e:back-phys}, for almost all $\w$, $\psi_E^\w$ satisfies for all vector fields $\phi\in H^1_{c,\Div}(\R^d)^d$ that vanish on $\Ic^\w$ and outside $D_R^\w$,
$$
\int_{D_R^\w} \nabla \phi:\nabla \psi_E^\w=0.
$$
We deduce that $\psi_E^\w$ is a weak solution of 
\begin{equation}\label{e.corr-mod}
\left\{\begin{array}{ll}
-\triangle\psi_E^\w+\nabla \bar \Sigma_E^\w=0,&\text{in $D_R^\w\setminus\Ic^\w$},\\
\Div\psi_E^\w=0,&\text{in $D_R^\w$},\\
\D(\psi_E^\w+Ex)=0,&\text{in $\Ic^\w\cap D_R^\w$},
\end{array}\right.
\end{equation}
in the sense of \cite[Definition~IV.1.1]{Galdi}.
Hence, by \cite[Lemma~IV.1.1]{Galdi},  there exists a unique pressure field $\bar \Sigma_E^\w \in \Ld^2(D_R^\w\setminus\Ic^\w)$
with the anchoring condition $\int_{2B} \bar \Sigma_E^\w\mathds1_{\R^d\setminus\Ic^\w}=0$,
such that~\eqref{e.corr-mod}
holds in the usual weak sense (that is, for all test functions $\phi \in H^1_0(D_R^\w \setminus \Ic^\w)^d$ without divergence-free constraint).
In addition, by~\cite[Theorems~IV.4.3 and~IV.5.2]{Galdi},  both $\psi_E^\w$ and $\bar \Sigma_E^\w$ are smooth in $B_{R/2} \setminus \Ic^\w$.
By the arbitrariness of $R$, this implies that the pressure field $\bar \Sigma_E^\w$ is well-defined in $\Ld^2_\loc(\R^d\setminus\Ic^\w)$ and that $\psi_E^\w$ and $\bar \Sigma_E^\w$ are smooth
on $\R^d \setminus \Ic^\w$. In particular, the solutions are classical and the boundary conditions in~\eqref{e.corr-eq}
are satisfied in a pointwise sense.
Note that the joint measurability of $\bar \Sigma_E$ on $\R^d\times\Omega$ easily follows from the reconstruction procedure for the pressure in~\cite{Galdi}; details are omitted.

\medskip
\substep{4.2} Proof that for all $R \ge 5$,
\begin{equation}\label{e.L2-pressure}
\fint_{B_R\setminus \Ic^\w} \Big(\bar \Sigma_E^\w-\fint_{B_R\setminus \Ic^\w}\bar \Sigma_E^\w\Big)^2 \,\lesssim\, \fint_{B_R\setminus \Ic^\w}|\nabla \psi_E^\w|^2.
\end{equation}
As usual for pressure estimates for the Stokes equation, we first need to construct a map $\zeta_R^\w \in H^1_0(B_R)$
such that
\begin{gather}
\Div \zeta_R^\w=\Big(\bar \Sigma_E^\w-\fint_{B_R\setminus \Ic^\w} \bar \Sigma_E^\w\Big)\mathds 1_{\R^d\setminus \Ic^\w},\label{eq:divzeta-ch}\\
\|\nabla \zeta_R^\w \|_{\Ld^2(B_R)} \,\lesssim\, \Big\|\bar \Sigma_E-\fint_{B_R\setminus \Ic^\w}\bar \Sigma_E^\w\Big\|_{\Ld^2(B_R\setminus \Ic^\w)},\label{eq:gradzeta-ch}
\end{gather}
with the slight twist that $\zeta_R^\w|_{I_n^\w}$ further needs to be constant for all $n$.
Testing~\eqref{e.corr-eq} with such a $\zeta_R^\w$ then yields
$$
\int_{\R^d\setminus \Ic^\w}  \nabla \zeta_R^\w : \nabla \psi_E^\w-\int_{\R^d\setminus \Ic^\w}\bar \Sigma_E^\w \Div \zeta_R^\w
\,=\,0,
$$
which entails in view of the choice~\eqref{eq:divzeta-ch} of $\zeta_R^\w$,
$$
\int_{B_R \setminus \Ic^\w}\Big|\bar \Sigma_E^\w-\fint_{B_R\setminus \Ic^\w}\bar \Sigma_E^\w\Big|^2 \,\le\, \int_{B_R \setminus\Ic^\w}  |\nabla \zeta_R^\w| |\nabla \psi_E^\w|,
$$
and~\eqref{e.L2-pressure} follows from~\eqref{eq:gradzeta-ch}.

\medskip\noindent
It remains to construct such a map $\zeta_R^\w$.
First define $\xi_R^\w\in H^1_0(B_R)^d$ (extended to zero outside $B_R$) as a solution of the divergence problem
\begin{gather*}
\Div \xi_{R}^\w\,=\, \Big(\bar \Sigma_E^\w-\fint_{B_R\setminus \Ic^\w} \bar \Sigma_E^\w\Big)\mathds 1_{\R^d \setminus \Ic^\w},\\
\|\nabla \xi_R^\w\|_{\Ld^2(B_R)}\,\lesssim\, \Big\|\bar \Sigma_E^\w-\fint_{B_R\setminus \Ic^\w}\bar \Sigma_E^\w\Big\|_{\Ld^2(B_R\setminus \Ic^\w)},
\end{gather*}
as provided by~\cite[Theorem~III.3.1]{Galdi}, where we emphasize that the multiplicative constant in the estimate is uniformly bounded in $R$.
Next, we need to modify $\xi_R^\w$ in the inclusions $I_n^\w$'s that intersect $B_R$ without changing $\Div \xi_R^\w$
and without increasing the norm of $\nabla \xi_R^\w$ too much. This is performed by constructing suitable compactly supported corrections around the inclusions.
For inclusions $I_n^\w$'s contained in $B_R$ with $\mathrm{dist}(I_n^\w, \partial B_R) \ge \delta$,
arguing as in Substep~3.1, we can construct a divergence-free vector field $\xi_{R,n}^\w\in H^1_0(I_n^\w+\frac\delta2B)^d$ that coincides with $-\xi_R^\w+\fint_{(I_n^\w+\frac\delta2B)\setminus I_n^\w}\xi_R^\w$ on $I_n^\w$ such that
\[\|\nabla\xi_{R,n}^\w\|_{\Ld^2((I_n^\w+\frac\delta2B)\setminus I_n^\w)}\lesssim\|\nabla\xi_R^\w\|_{\Ld^2((I_n^\w+\frac\delta2B)\setminus I_n^\w)}.\]
We turn to inclusions $I_n^\w$'s that intersect $B_R$ such that $\mathrm{dist}(I_n^\w,\partial B_R) < \delta$, for which we construct a divergence-free vector field $\xi_{R,n}^\w \in H^1_0(B_R\cap(I_n^\w+\frac\delta2B))^d$ that coincides with $-\xi_R^\w$ on $B_R\cap I_n^\w$ (that is indeed divergence-free there).
Such a vector field can be constructed as an application of the Bogovskii operator on $B_R\cap(I_n^\w+\frac\delta2B)\setminus I_n^\w$, in view of the compatibility condition
\begin{eqnarray}
\int_{B_R\cap(I_n^\w+\frac\delta2B)\setminus I_n^\w} \Div \xi_{R,n}^\w &=& \int_{\partial(B_R\cap(I_n^\w+\frac\delta2B))} \xi_{R,n}^\w \cdot \nu - \int_{\partial(B_R\cap I_n^\w)} \xi_{R,n}^\w \cdot \nu\nonumber
\\
&=&\int_{\partial (B_R\cap I_n^\w)} \xi_R^\w \cdot \nu
\,=\,\int_{B_R\cap I_n^\w} \Div \xi_R^\w \,=\,0,\label{e.compatibility}
\end{eqnarray}
and it satisfies
\[\|\nabla\xi_{R,n}^\w\|_{\Ld^2(B_R\cap(I_n^\w+\frac\delta2B)\setminus I_n^\w)}\lesssim\|\xi_R^\w\|_{H^\frac12(\partial(B_R\cap I_n^\w))}.\]
Hence, by a trace estimate (with $\partial I_n^\w$ at distance at most $\delta$ from $\partial B_R$, on which $\xi_R^\w$ vanishes) and Poincar\'e's inequality,
\begin{eqnarray*}
\|\nabla\xi_{R,n}^\w\|_{\Ld^2(B_R\cap(I_n^\w+\frac\delta2B)\setminus I_n^\w)}
&\lesssim& \|\xi_R^\w\|_{H^1( B_R\cap (I_n^\w+2\delta B))}
\\
&\lesssim&\|\nabla\xi_R^\w\|_{\Ld^2(B_R\cap (I_n^\w+2\delta B))}.
\end{eqnarray*}
We finally define
$$
\zeta_R^\w\,:=\, \xi_R^\w +\sum_{n:I_n^\w\cap B_R\ne\varnothing}\xi_{R,n}^\w,
$$
which by construction is constant in each of the inclusions $I_n^\w$'s and satisfies the required properties~\eqref{eq:divzeta-ch} and~\eqref{eq:gradzeta-ch}.

\medskip

\substep{4.3} Extension of $\bar \Sigma_E$ to $\R^d$ and estimate of $\nabla \bar \Sigma_E$.\\
In this substep, we extend $\bar \Sigma_E$ to $\R^d$ in such a way that  $\bar \Sigma_E \in \Ld^2(\Omega;H^1_\loc(\R^d))$, that $\nabla \bar \Sigma_E$ is stationary, and
that we have for all $R\ge 5$,
\begin{equation}\label{e.bd-grad-p}
\expec{\fint_{B_R} \Big(\bar \Sigma_E - \fint_{B_R} \bar \Sigma_E\Big)^2}+\expec{|\nabla \bar \Sigma_E|^2} \,\lesssim\, |E|^2.
\end{equation}
We start by proving that $(\nabla \bar \Sigma_E^\w) \mathds 1_{\R^d \setminus \Ic}$ is a stationary field and satisfies
\begin{equation}\label{eq:bound-pE1}
\expecm{|(\nabla \bar \Sigma_E) \mathds 1_{\R^d \setminus \Ic}|^2} \,\lesssim\, |E|^2.
\end{equation}
By the Stokes equation in form of $(\triangle \psi_E^\w) \mathds 1_{\R^d \setminus \Ic}=(\nabla \bar \Sigma_E^\w) \mathds 1_{\R^d \setminus \Ic^\w}$, it suffices to prove that $(\triangle \psi_E)\mathds 1_{\R^d\setminus \Ic} \in \Ld^2(\Omega)$ satisfies $\expecm{|(\triangle \psi_E)\mathds 1_{\R^d\setminus \Ic}|^2}\,\lesssim\,|E|^2$.
Since $(\nabla \psi_E)\mathds 1_{\R^d\setminus \Ic}$ is stationary and since $\psi_E$ is of class $C^2$ up to the boundaries $\partial I_n$, it is enough to prove that for almost all $\w$,
\begin{equation}\label{eq:+refrev}
\limsup_{R\uparrow \infty} \fint_{B_R} |(\nabla^2 \psi_E^\w)\mathds 1_{\R^d\setminus \Ic}|^2 \,\lesssim\, |E|^2.
\end{equation}
To this aim, it suffices to show that for all $x \in \R^d$,
\begin{equation}\label{e.reg-th-corr-galdi}
\int_{B_{\delta/8}(x)}(|\nabla^2 \psi_E^\w|^2+|\nabla \Sigma_E^\w|^2)\mathds1_{\R^d\setminus \Ic}
\,\lesssim_\delta\, \int_{B_{5}(x)}|\nabla \psi_E^\w|^2,
\end{equation}
since the desired estimate~\eqref{eq:+refrev} then follows in combination with the ergodic theorem and the bound $\expec{|\nabla \psi_E|^2}\lesssim |E|^2$.
First consider the case when $x\in \R^d$ satisfies $\mathrm{dist}\,(x,\Ic^\w)>\delta/4$, for which $B_{\delta/4}(x) \subset \R^d \setminus \Ic^\w$.
By interior regularity for the Stokes equation in form of~\cite[Theorems~IV.4.1]{Galdi}, by~\eqref{e.L2-pressure}, and by Poincar\'e's inequality, we then have
with the choice $c_1^\w=\fint_{B_{\delta/2}(x)} \psi_E^\w$ and $c_2^\w= \fint_{B_5(x)\setminus \Ic^\w}\bar \Sigma_E^\w$,
\begin{eqnarray*}
\int_{B_{\delta/8}(x)}(|\nabla^2 \psi_E^\w|^2+|\nabla\bar \Sigma_E^\w|^2)\mathds1_{\R^d\setminus \Ic^\w}
&\lesssim_\delta &\int_{B_{\delta/4}(x)}(|\nabla \psi_E^\w|^2+|\psi_E^\w-c_1^\w|^2+|\bar \Sigma_E^\w-c_2^\w|^2)
\\
&\lesssim& \int_{B_{5}(x)}|\nabla \psi_E^\w|^2 ,
\end{eqnarray*}
that is,~\eqref{e.reg-th-corr-galdi}.
Next consider the case when $x \in \R^d$ satisfies $\mathrm{dist}\,(x,\Ic^\w)\le\delta/4$, and let~$I_n^\w$ be the unique ball such that  $\mathrm{dist}\,(x,I_n^\w)<\delta/4$. 
By the boundary regularity theory for the Stokes equation in form of~\cite[Theorems~IV.5.1--5.3]{Galdi}, we then have 
with the choice $c_1^\w=\fint_{I_n^\w+\frac \delta 2 B} \psi_E^\w$ and $c_2^\w=\fint_{B_5(x)\setminus \Ic^\w}\bar \Sigma_E^\w$,
\begin{multline*}
\int_{B_{\delta/8}(x)}(|\nabla^2 \psi_E^\w|^2+|\nabla \bar \Sigma_E^\w|^2)\mathds1_{\R^d\setminus \Ic^\w}
\,\lesssim_\delta\,\|\psi_E^\w|_{I_n^\w}-c_1^\w\|_{H^\frac32(\partial I_n^\w)}^2
\\+\|\bar \Sigma_E^\w-c_2^\w\|_{\Ld^2((I_n^\w+\frac\delta2B)\setminus I_n^\w)}^2+\|\psi_E^\w-c_{1}^\w\|_{H^1((I_n^\w+\frac\delta2B)\setminus I_n^\w)}.
\end{multline*}
Since $\psi_E^\w$ is affine on $I_n^\w$, we have
\begin{eqnarray*}
\|\psi_E^\w|_{I_n^\w}-c_{1}^\w\|_{H^\frac32(\partial I_n^\w)}
&\lesssim& \|\psi_E^\w|_{I_n^\w}-c_{1}^\w\|_{H^2( I_n^\w)}=\|\psi_E^\w -c_{1}^\w\|_{H^1( I_n^\w)}
\\
&\le&\|\psi_E^\w -c_{1}^\w\|_{H^1( I_n^\w+\frac\delta 2 B)},
\end{eqnarray*}
while Poincaré's inequality with mean-value zero yields
\[\|\psi_E^\w-c_1^\w\|_{H^1(I_n^\w+\frac\delta2B)}\,\lesssim_\delta\, 
\|\nabla \psi_E^\w\|_{\Ld^2(I_n^\w+\frac\delta2B)},\]
so that in combination with  \eqref{e.L2-pressure} the above turns into \eqref{e.reg-th-corr-galdi}.

\medskip\noindent
It remains to extend $\bar \Sigma_E$ on the inclusions.
We simply choose $\bar \Sigma_E^\w |_{B_{1/2}(x_n^\w)}=\fint_{(I_n^\w+\frac\delta2B)\setminus I_n^\w} \bar \Sigma_E^\w$, 
and we extend $\bar \Sigma_E^\w$ radially linearly between $\partial I_n^\w$ and $\partial B_{1/2}(x_n^\w)$ (recall that $\bar \Sigma_E^\w \mathds 1_{\R^d\setminus \Ic^\w}$ is continuous up to the boundary).
So defined, $\bar \Sigma_E^\w$ belongs to $H^1_\loc(\R^d)$ and $\nabla \bar \Sigma_E$ is stationary on~$\R^d$.
We conclude by establishing~\eqref{e.bd-grad-p}.
Noting that the choice of the extension ensures
\[\int_{I_n^\w} |\nabla \Sigma_E^\w|^2 \,\lesssim_\delta\, \int_{(I_n^\w+\frac\delta4B)\setminus I_n^\w} |\nabla \Sigma_E^\w|^2,\]
the gradient estimate in~\eqref{e.bd-grad-p} simply follows from~\eqref{eq:bound-pE1},
and it remains to check the other part.
By the definition of the extension, with $c^\w=\fint_{B_{R+2}\setminus \Ic} \bar \Sigma_E^\w$, we find
using~\eqref{e.L2-pressure} and~\eqref{eq:bound-pE1},
\begin{eqnarray*}
\lefteqn{\fint_{B_R} \Big(\bar \Sigma_E^\w - \fint_{B_R} \bar \Sigma_E^\w\Big)^2\,\lesssim\, \fint_{B_R} (\bar \Sigma_E^\w - c^\w)^2}\\
 &\lesssim& 
\fint_{B_R} (\bar \Sigma_E^\w - c^\w)^2\mathds1_{\R^d\setminus\Ic^\w}+R^{-d} \int_{B_R\cap \Ic} (\bar \Sigma_E^\w-c^\w)^2
\\
&\lesssim_\delta&  \fint_{B_R} (\bar \Sigma_E^\w - c^\w)^2\mathds1_{\R^d\setminus\Ic^\w}+R^{-d} \sum_{n:I_n^\w \cap B_R\ne\varnothing}  \int_{(I_n^\w+\frac\delta2B)\setminus I_n^\w}( |\bar \Sigma_E^\w-c^\w|^2+|\nabla\bar \Sigma_E^\w|^2)
\\
& \lesssim &
  \fint_{B_{R+2}} (\bar \Sigma_E^\w - c^\w)^2\mathds1_{\R^d\setminus\Ic^\w}+\fint_{B_{R+2}} |\nabla\bar \Sigma_E^\w|^2\mathds1_{\R^d\setminus\Ic^\w}
\\
&\lesssim & \fint_{B_{R+2}} |\nabla \psi_E^\w|^2,
\end{eqnarray*}
and the estimate~\eqref{e.bd-grad-p} follows.

\medskip

\substep{4.4} Construction of a stationary pressure field $\Sigma_E$.\\
Let $\chi\in C^\infty_c(B)$ satisfy $\int_B\chi=1$, consider the rescaled kernel $\chi_r=\frac1{r^d} \chi(\frac \cdot r)$ for $r\ge1$, and define $P_r:=\bar \Sigma_E - \chi_r \ast\bar \Sigma_E$.
By construction, $P_r$ is stationary, and we claim that 
\begin{eqnarray}\label{e.bd-const-press1}
\expec{P_r^2+|\nabla P_r|^2} &\lesssim& |E|^2,
\\
\lim_{r \uparrow \infty} \expecm{|\nabla P_r-\nabla \bar \Sigma_E|^2\mathds1_{\R^d \setminus \Ic}} &=&0.\label{e.bd-const-press2}
\end{eqnarray}
From \eqref{e.bd-const-press1}, we deduce by weak compactness that there exists some $\tilde P \in H^1(\Omega)$ such that $(\tilde P_r,\nabla \tilde P_r) \rightharpoonup (\tilde P,\nabla \tilde P)$ weakly
in $\Ld^2(\Omega)$ along some subsequence (not relabelled), with
\begin{equation}\label{e.bd-const-press3}
\expecm{P^2+|\nabla P|^2}\,\lesssim \,|E|^2.
\end{equation}
From \eqref{e.bd-const-press2} and the weak lower-semicontinuity of the $\Ld^2(\Omega)$-norm, we then deduce 
$$
\expecm{|\nabla P-\nabla \bar \Sigma_E|^2\mathds1_{\R^d \setminus \Ic}} \,\le\,\liminf_{r\uparrow \infty} \expecm{|\nabla P_r-\nabla \bar \Sigma_E|^2\mathds1_{\R^d \setminus \Ic}} \,=\,0.
$$
Hence, for almost all $\w$, the limit $P^\w$ coincides with $\bar \Sigma_E^\w$ up to an additive constant on the connected set $\R^d\setminus\Ic^\w$.
We then define the stationary pressure as $\Sigma_E \mathds1_{\R^d \setminus \Ic}:=\big(P-\expecm{P \mathds1_{\R^d \setminus \Ic}}\big) \mathds1_{\R^d \setminus \Ic}$, which satisfies $\expecm{\Sigma_E\mathds1_{\R^d \setminus \Ic}}=0$ and the a priori estimate~(ii).

\medskip\noindent
It remains to give the arguments in favor of \eqref{e.bd-const-press1} and \eqref{e.bd-const-press2}.
We start with the former.
For all $R\ge r\ge1$, for $c^\w=\fint_{B_{R+r}}\bar \Sigma_E^\w$, we have
\begin{eqnarray*}
\fint_{B_R} (P_r^\w)^2+|\nabla P_r^\w|^2&=& \fint_{B_R} (\bar \Sigma_E^\w -\chi_r \ast\bar \Sigma_E^\w)^2+|\nabla \bar \Sigma_E^\w-\chi_r\ast\nabla \bar P^\w_E|^2
\\
&\lesssim&  \fint_{B_R} (\bar \Sigma_E^\w-c^\w)^2+(\chi_r \ast( \bar \Sigma_E^\w-c^\w))^2+ |\nabla \bar \Sigma_E^\w|^2+| \chi_r\ast\nabla \bar \Sigma_E^\w|^2 
\\
&\lesssim& R^{-d} \int_{B_{R+r}} |\nabla \bar \Sigma_E^\w|^2+(\bar \Sigma_E^\w-c^\w)^2.
\end{eqnarray*}
Taking the expectation and using \eqref{e.bd-grad-p} then yields by stationarity of $P_r$,
$$
\expec{ P_r^2+|\nabla P_r|^2} \,\lesssim\,  \frac{(R+r)^d}{R^d} |E|^2,
$$
from which  \eqref{e.bd-const-press1} follows by taking the limit $R\uparrow \infty$.
We turn to~\eqref{e.bd-const-press2}.
By definition of~$P_r$ and
since $|\nabla \chi_r| \lesssim \frac1{r^{d+1}} \mathds1_{B_R}$ for all $R\ge r\ge1$, we have for $c^\w=\fint_{B_{R+r}} \bar \Sigma_E^\w$,
\begin{multline*}
\fint_{B_R} |\nabla P_r^\w-\nabla \bar \Sigma_E^\w|^2 \mathds 1_{\R^d \setminus \Ic^\w}\,\le\,\fint_{B_R}  |\nabla \chi_r * \bar \Sigma_E^\w|^2 
\,=\,\fint_{B_R}  |\nabla \chi_r * (\bar \Sigma_E^\w-c^\w)|^2\\
\,\lesssim\,\frac1r\frac{(R+r)^d}{R^d}  \fint_{B_{R+r}}   (\bar \Sigma_E^\w-c^\w)^2.
\end{multline*}
As before, taking the expectation, recalling that $(\nabla P_r-\nabla\bar \Sigma_E)\mathds 1_{\R^d \setminus \Ic}$ is stationary, using~\eqref{e.bd-grad-p}, and letting $R\uparrow \infty$, we deduce
$$
\expecm{ |\nabla P_r-\nabla \bar \Sigma_E|^2 \mathds 1_{\R^d \setminus \Ic}} \,\lesssim\,  \frac1r   |E|^2,
$$
from which the claim  \eqref{e.bd-const-press2} follows.

\medskip

\substep{4.5} Proof of existence and uniqueness for (i)--(ii).\\
In Step~1, we have shown that if $\psi_E$ is a solution of (i)--(ii), then $\Psi_E=\nabla  \psi_E$ satisfies the abstract problem~\eqref{e.weak-form-corr},
for which existence and uniqueness is proved in Step~2. 
In Step~3, we considered the unique solution $\Psi_E$ of~\eqref{e.weak-form-corr} and proved that $\Psi_E=\nabla \psi_E$ is automatically a weak solution of~\eqref{e.corr-eq} in form of~\eqref{e:back-phys}.
In Substeps~4.1--4.4, we reconstructed a unique stationary pressure field $\Sigma_E$ (with vanishing expectation) such that $\psi_E$ is a classical solution of~\eqref{e.corr-eq}. Uniqueness for~(i)--(ii) then follows from uniqueness for~\eqref{e.weak-form-corr}.
For the existence part for~(i)--(ii), it remains to note that $\Sigma_E$ and $\psi_E$ satisfy (ii) as shown in Substep~4.4. %

\medskip

\step5 Proof of~(iii)--(iv).\\
The convergences in (iii) are a standard application of the ergodic theorem.
The sublinearity (iv) of the corrector $\psi_E^\w$ at infinity is also a standard result  for random fields the gradients of which are stationary and have vanishing expectation, cf.~\cite{PapaVara,JKO94}. 
\end{proof}

\bigskip
\section{Proof of the homogenization result}
This section is devoted to the proof of Theorem~\ref{th:Stokes}, making use of the correctors $(\psi_E)_E$ defined in Proposition~\ref{prop:corr-Stokes} and adapting the classical oscillating test function method by Tartar~\cite{Tartar-09}. We split the proof into eight different steps.

\medskip
\step1 Reformulation of the equations.\\
We show that the solution $u_\e^\w$ of~\eqref{eq:Stokes} satisfies in the weak sense in the whole domain $U$,
\begin{equation}\label{eq:Stokes-re}
-\triangle u_\e^\w+\nabla (P_\e^\w\mathds1_{U\setminus\Ic^\w_\e(U)})=f\mathds1_{U\setminus\Ic_\e^\w(U)}-\sum_{n\in\Nc_\e^\w(U)}\delta_{\e\partial I_n^\w}\sigma(u_\e^\w,P_\e^\w)\nu,
\end{equation}
while the corrector $\psi_E^\w$ satisfies in the whole space $\R^d$,
\begin{equation}\label{eq:corr-re}
-\triangle \psi_E^\w+\nabla (\Sigma_E^\w\mathds1_{\R^d\setminus\Ic^\w})=-\sum_n\delta_{\partial I_n^\w}\sigma(\psi_E^\w+Ex,\Sigma_E^\w)\nu.
\end{equation}
We focus on~\eqref{eq:Stokes-re}, and leave  the proof of~\eqref{eq:corr-re} (which is similar) to the reader.
Since $u_\e^\w$ is divergence-free, an integration by parts yields
\[\int_{U}\nabla\zeta:\nabla u_\e^\w\,=\,2\int_{U}\nabla\zeta:\D(u_\e^\w),\]
and thus, since $\D(u_\e^\w)=0$ in $\Ic_\e^\w(U)$,
\[\int_{U}\nabla\zeta:\nabla u_\e^\w\,=\,2\int_{U\setminus\Ic_\e^\w(U)}\nabla\zeta:\D(u_\e^\w).\]
Recalling the definition $\sigma(u_\e^\w,P_\e^\w)=2\D(u_\e^\w)-P_\e^\w\Id$,
integrating by parts, and using equation~\eqref{eq:Stokes}, we obtain
\begin{eqnarray*}
\int_{U}\nabla\zeta:\nabla u_\e^\w-\int_{U\setminus\Ic_\e^\w(U)}(\Div\zeta)\,P_\e^\w
&=&\int_{U\setminus\Ic_\e^\w(U)}\nabla\zeta:\sigma(u_\e^\w,P_\e^\w)\\
&=&\int_{U\setminus\Ic_\e^\w(U)}\zeta\cdot f-\sum_{n\in\Nc_\e^\w(U)}\int_{\e\partial I_n^\w}\zeta\cdot\sigma(u_\e^\w,P_\e^\w)\nu,
\end{eqnarray*}
that is,~\eqref{eq:Stokes-re}.

\medskip
\step2 Energy estimates.\\
We now show that for almost all $\w$ the solution $u_\e^\w$ of~\eqref{eq:Stokes} satisfies
\begin{equation}\label{e.energy-estim}
\int_U|\nabla u_\e^\w|^2+\int_{U\setminus\Ic_\e^\w(U)}|P_\e^\w|^2\,\lesssim_\delta\,\int_U|f|^2.
\end{equation}
For almost all $\w$, by weak compactness, this allows us to consider $\bar u^\w\in H^1_0(U)^d$ and $\bar Q^\w\in\Ld^2(U)$ such that, along a subsequence (not relabelled) as $\e\downarrow0$,
\begin{equation}\label{eq:conv-up}
u_\e^\w\cvf\bar u^\w\quad\text{ in $H^1_0(U)$,}\qquad\text{and}\qquad P_\e^\w\mathds1_{U\setminus\Ic_\e^\w(U)}\cvf\bar Q^\w\quad\text{ in $\Ld^2(U)$.}
\end{equation}
In particular, by Rellich's theorem, $u_\e^\w\to\bar u^\w$ in $\Ld^2(U)$ strongly.

\medskip\noindent
Here comes the argument for~\eqref{e.energy-estim}. For all $v\in H^1_0(U)$ with $\Div v=0$ in $U$ and with $\D(v)=0$ in $\Ic_\e(U)$, testing the formulation~\eqref{eq:Stokes-re} of the Stokes equation with $v$ yields
\[\int_{U}\nabla v:\nabla u_\e^\w=\int_{U\setminus\Ic_\e^\w(U)}v\cdot f,\]
which for the choice $v=u_\e^\w$ yields
\begin{equation}\label{eq:energy}
\int_{U}|\nabla u_\e^\w|^2\,=\,\int_{U\setminus\Ic_\e^\w(U)}f\cdot u_\e^\w \,\lesssim\, \Big( \int_{U}|f|^2\Big)^\frac12 \Big( \int_{U}|\nabla u_\e^\w|^2\Big)^\frac12
\end{equation}
by Poincar\'e's inequality in $H^1_0(U)$, that is,~\eqref{e.energy-estim} for $\nabla u_\e^\w$.
The corresponding estimate for the pressure is obtained by
a similar argument as in  Substep~4.2 of the proof of Proposition~\ref{prop:corr-Stokes}.

\medskip
\step3 A priori estimates at inclusion boundaries.\\
We claim that the solution $u_\e^\w$ of~\eqref{eq:Stokes} and the corrector $\psi_E^\w$ satisfy for almost all $\w$,
\begingroup\allowdisplaybreaks
\begin{eqnarray}
\sum_{n\in\Nc_\e^\w(U)}\int_{\e\partial I_n^\w}|u_\e^\w|^2&\lesssim&\frac1\e\int_U|f|^2,\label{eq:est-sum-par-pu}\\
\sum_{n\in\Nc_\e^\w(U)}\int_{\e\partial I_n^\w}|\nabla u_\e^\w|^2+|P_\e^\w|^2&\lesssim_\delta&\frac 1\e\int_U|f|^2,\label{eq:est-sum-par-u}\\
\sum_{n\in\Nc_\e^\w(U)}\int_{\e\partial I_n^\w}|\psi_E^\w(\tfrac\cdot\e)|^2&\lesssim_\delta&\frac1\e\int_U|\psi_E^\w(\tfrac\cdot\e)|^2,\label{eq:est-sum-par-cor}\\
\sum_{n\in\Nc_\e^\w(U)}\int_{\e\partial I_n^\w}|\nabla\psi_{E}^\w(\tfrac\cdot\e)|^2+|\Sigma_E^\w(\tfrac\cdot\e)|^2&\lesssim_\delta&\frac1\e\int_U|\nabla\psi_E^\w(\tfrac\cdot\e)|^2+|(\Sigma_E^\w\mathds1_{\R^d\setminus\Ic^\w})(\tfrac\cdot\e)|^2.~~~\label{eq:est-sum-par-cor+}
\end{eqnarray}
\endgroup
We start with the proof of~\eqref{eq:est-sum-par-pu}. For all $n\in\Nc_\e^\w(U)$, since $u_\e^\w$ is affine in $\e I_n^\w$, there holds
\[\int_{\e\partial I_n^\w}|u_\e^\w|^2\,\lesssim\,\frac1\e\int_{\e I_n^\w}|u_\e^\w|^2,\]
so that
\[\sum_{n\in\Nc_\e^\w(U)}\int_{\e\partial I_n^\w}|u_\e^\w|^2\,\lesssim\,\frac1\e\int_{U}|u_\e^\w|^2,\]
and the claim~\eqref{eq:est-sum-par-pu} follows from Poincaré's inequality and~\eqref{e.energy-estim}.
Likewise, for all $n$, since $\psi_E^\w$ is affine in $I_n^\w$, we find
\[\int_{\partial I_n^\w}|\psi_E^\w|^2\lesssim\int_{I_n^\w}|\psi_E^\w|^2,\]
and the claim~\eqref{eq:est-sum-par-cor} follows after summing and rescaling.
We turn to the proof of~\eqref{eq:est-sum-par-u}.
By scaling, it suffices to check that $\hat u_\e^\w:=\e^{-2}u_\e^\w(\e\cdot)$ and $\hat P_\e^\w:=\e^{-1} P_\e^\w(\e\cdot)$ satisfy
\begin{equation}\label{eq:est-sum-par-u+}
\sum_{n\in\Nc_\e^\w(U)}\int_{\partial I_n^\w}|\nabla\hat u_\e^\w|^2+|\hat P_\e^\w|^2\lesssim_\delta\, \frac1{\e^2} \int_{\frac1\e U}|f(\e\cdot)|^2.
\end{equation}
Given $n\in\Nc_\e^\w(U)$, a trace estimate yields
\[\int_{\partial I_n^\w}|\nabla\hat u_\e^\w|^2+|\hat P_\e^\w|^2\lesssim_\delta\|(\nabla\hat u_\e^\w,\hat P_\e^\w)\|_{H^1((I_n^\w+\frac\delta4B)\setminus I_n^\w)}^2.\]
Recalling that the inclusion~$I_n^\w$ is at distance at least $\delta>0$ from other inclusions and from $\frac1\e\partial U$ so that $-\triangle \hat u_\e^\w+\nabla\hat P_\e^\w=f(\e\cdot)$ is satisfied in the annulus $(I_n^\w+\delta B)\setminus I_n^\w$, the regularity theory for the Stokes equation near a boundary in form of~\cite[Theorems~IV.5.1--5.3]{Galdi} leads to the following, with $c_{n,\e}^\w:=\fint_{I_n^\w+\frac\delta2 B}\hat u_\e^\w$,
\begin{multline*}
\int_{\partial I_n^\w}|\nabla\hat u_\e^\w|^2+|\hat P_\e^\w|^2
\lesssim_\delta\|\hat u_\e^\w|_{I_n^\w}-c_{n,\e}^\w\|_{H^\frac32(\partial I_n^\w)}^2+\|f(\e\cdot)\|_{\Ld^2(I_n^\w+\frac\delta2B)}^2
\\+\|\hat P_\e^\w\|_{\Ld^2((I_n^\w+\frac\delta2B)\setminus I_n^\w)}^2+\|\hat u_\e^\w-c_{n,\e}^\w\|_{H^1((I_n^\w+\frac\delta2B)\setminus I_n^\w)}.
\end{multline*}
Since $\hat u_\e^\w$ is affine on $I_n^\w$, we have
\[\|\hat u_\e^\w|_{I_n^\w}-c_{n,\e}^\w\|_{H^\frac32(\partial I_n^\w)}^2\,\lesssim\, \|\hat u_\e^\w-c_{n,\e}^\w\|_{H^2( I_n^\w)}^2 \,=\,\|\hat u_\e^\w -c_{n,\e}^\w\|_{H^1( I_n^\w)}^2,\]
while Poincaré's inequality with mean-value zero yields
\[\|\hat u_\e^\w-c_{n,\e}^\w\|_{H^1(I_n^\w+\frac\delta2B)}\,\lesssim\, 
\|\nabla \hat u_\e^\w\|_{\Ld^2(I_n^\w+\frac\delta2B)}^2,\]
so that the above turns into
\begin{equation*}
\int_{\partial I_n^\w}|\nabla\hat u_\e^\w|^2+|\hat P_\e^\w|^2
\lesssim_\delta\|f(\e\cdot)\|_{\Ld^2(I_n^\w+\frac\delta2B)}^2+\|(\nabla\hat u_\e^\w,\hat P_\e^\w\mathds1_{\R^d\setminus\Ic^\w})\|_{\Ld^2(I_n^\w+\frac\delta2B)}^2.
\end{equation*}
Since the balls of the collection $\{I_n^\w+\frac\delta2B\}_n$ are all disjoint, the rescaled version of the energy estimate~\eqref{e.energy-estim} leads to
\begin{equation*}
\sum_{n\in\Nc_\e^\w(U)}\int_{\partial I_n^\w}|\nabla\hat u_\e^\w|^2+|\hat P_\e^\w|^2\lesssim_\delta\int_{\frac1\e U}|f(\e\cdot)|^2+|\nabla\hat u_\e^\w|^2+|\hat P_\e^\w\mathds1_{\R^d\setminus\Ic^\w}|^2\, \lesssim \, \frac1{\e^2} \int_{\frac1\e U}|f(\e\cdot)|^2,
\end{equation*}
that is,~\eqref{eq:est-sum-par-u+}.
It remains to establish~\eqref{eq:est-sum-par-cor+}.
Applying as above a trace estimate together with the regularity theory for the Stokes equation near a boundary (cf. Substep~4.3
in the proof of Proposition~\ref{prop:corr-Stokes}), we obtain
\begin{equation*}
\sum_{n\in\Nc_\e^\w(U)}\int_{\partial I_n^\w}|\nabla\psi_E^\w|^2+|\Sigma_E^\w|^2
\lesssim_\delta\int_{\frac1\e U}|\nabla\psi_E^\w|^2+|\Sigma_E^\w\mathds1_{\R^d\setminus\Ic^\w}|^2,
\end{equation*}
and the claim~\eqref{eq:est-sum-par-cor+} follows after rescaling.

\medskip
\step4 Oscillating test function method.\\
We show that for all test functions $\bar v\in C^\infty_c(U)^d$ with $\Div\bar v=0$ we have for almost all $\omega$, along a subsequence (not relabelled),
\begin{equation}\label{eq:oscillating-test-fct-res}
2\int_U\D(\bar v):\D(\bar u^\w)+\lim_{\e\downarrow0}\sum_{E\in\Ec}2\int_U(\nabla_E\bar v)\,\D(\psi_E^\w)(\tfrac\cdot\e):\D(u_\e^\w)= (1-\lambda)\int_U\bar v\cdot f,
\end{equation}
where the sum runs over an orthonormal basis $\Ec$ of symmetric trace-free matrices $\Md_0^\Sym$, and where the limit in the left-hand side indeed exists (and is computed in the next step).

\medskip\noindent
Let a typical $\omega\in\Omega$ be fixed such that the bounds of Steps~1--2 hold as well as the convergence~\eqref{eq:conv-up} along a subsequence (not relabelled),
and such that for all $E\in\Md_0$ the corrector $\psi_E^\w$ and corresponding pressure $\Sigma_E^\w$ satisfy
the corrector equation~\eqref{eq:corr-Stokes} in the classical sense as well as the properties~(iii)--(iv) of Proposition~\ref{prop:corr-Stokes}.
Given a test function $\bar v\in C^\infty_c(U)^d$ with $\Div\bar v=0$, we follow Tartar's ideas and define its oscillatory version $v_\e^\w\in H^1_0(U)^d$ via
\[v_\e^\w:=\bar v+\sum_{E\in\Ec}\e\psi_E^\w(\tfrac\cdot\e)\nabla_E\bar v,\]
where we recall the notation $\nabla_E\bar v=E:\nabla\bar v$.
Testing equation~\eqref{eq:Stokes-re} with $v_\e^\w$ leads to
\begin{multline}\label{eq:test-osc}
\int_U\nabla v_\e^\w:\nabla u_\e^\w-\int_{U\setminus\Ic_\e^\w(U)}(\Div v_\e^\w)\,P_\e^\w\\
=\int_{U\setminus\Ic_\e^\w(U)}v_\e^\w\cdot f
-\sum_{n\in\Nc_\e^\w(U)}\int_{\e\partial I_n^\w}v_\e^\w\cdot\sigma(u_\e^\w,P_\e^\w)\nu,
\end{multline}
and it remains to examine each of the four terms appearing in this identity.
\begin{enumerate}[$\bullet$]
\item First, an integration by parts with $\Div u_\e^\w=0$ yields
\[\quad\int_U\nabla v_\e^\w:\nabla u_\e^\w\,=\,2\int_U\D(v_\e^\w):\D(u_\e^\w),\]
and then inserting the definition of $v_\e^\w$,
\begin{multline*}
\quad\int_U\nabla v_\e^\w:\nabla u_\e^\w=2\int_U\D(\bar v):\D(u_\e^\w)+\sum_{E\in\Ec}2\int_U\e\psi_E^\w(\tfrac\cdot\e)\otimes\nabla\nabla_E\bar v:\D(u_\e^\w)\\
+\sum_{E\in\Ec}2\int_U(\nabla_E\bar v)\,\D(\psi_E^\w)(\tfrac\cdot\e):\D(u_\e^\w).
\end{multline*}
By Step~2, the first right-hand side term converges to $2\int_U\D(\bar v):\D(\bar u^\w)$. By sublinearity of~$\psi_E^\w$ (cf.\@ Proposition~\ref{prop:corr-Stokes}(iv)), together with the boundedness of $\nabla u_\e^\w$ in $\Ld^2(U)$ (cf.~Step~2), the second right-hand side term converges to $0$.
\smallskip\item Second, the definition of $v_\e^\w$ with $\Div\bar v=0$ and $\Div\psi_E^\w=0$ leads to
\[\quad\int_{U\setminus\Ic_\e^\w(U)}(\Div v_\e^\w)\,P_\e^\w=\sum_{E\in\Ec}\int_{U\setminus\Ic_\e^\w(U)}\e\psi_E^\w(\tfrac\cdot\e)\cdot P_\e^\w\nabla\nabla_E\bar v,\]
which converges to $0$ in view of the sublinearity of $\psi_E^\w$ (cf.~Proposition~\ref{prop:corr-Stokes}(iv)) together with the boundedness of $P_\e^\w$ in $\Ld^2(U)$ (cf.~Step~2).
\smallskip\item Third, the sublinearity of $\psi_E^\w$ (cf.~Proposition~\ref{prop:corr-Stokes}(iv)) implies $v_\e^\w \to\bar v $ in $\Ld^2(U)$ and the ergodic theorem for the inclusion process yields $\mathds{1}_{U \setminus \Ic^\w_\e(U)}\cvf\expecm{\mathds{1}_{\R^d\setminus\Ic}}\mathds1_U=(1-\lambda)\mathds1_U$ weakly-* in $\Ld^\infty(U)$ for typical~$\w$, so that $\int_{U\setminus\Ic^\w_\e(U)}v_\e^\w\cdot f\to(1-\lambda)\int_{U}\bar v\cdot f$.
\smallskip\item Fourth, for $n\in\Nc_\e^\w(U)$, 
the oscillating test function $v_\e^\w$ can be expanded as follows, for all $x \in \e\partial I_n^\w$,
\begin{multline*}
\quad\Big|v_\e^\w(x)- \bar v(\e x_n^\w)-\nabla \bar v(\e x_n^\w)\,(x-\e x_n^\w)-\sum_{E\in\Ec}\e\psi_E(\tfrac x\e)\nabla_E \bar v(\e x_n^\w) \Big|\\
\lesssim \e^2\|\nabla^2 \bar v\|_{\Ld^\infty} \max_{E\in\Ec}(1+|\psi_E(\tfrac x\e)|).
\end{multline*}
Setting for abbreviation $\Theta_{\e,n}^\w:=(\nabla\bar v-\D(\bar v))(\e x_n^\w)\in\Md^\Skew$, and recalling the choice $\Tr\D(\bar v)=\Div\bar v=0$, this can be reorganized as
\begin{multline*}
\quad\Big|v_\e^\w(x)
-\bar v(\e x_n^\w)-\Theta_{\e,n}^\w(x-\e x_n^\w)
- \sum_{E\in\Ec}\e\big( \psi_E^\w(\tfrac x\e)+E (\tfrac x\e-x_n^\w)\big)\nabla_E \bar v(\e x_n^\w)\Big|\\
\lesssim\e^2\|\nabla^2 \bar v\|_{\Ld^\infty} \max_{E\in\Ec}(1+|\psi_E(\tfrac x\e)|).
\end{multline*}
Inserting this approximation of $v_\e^\w$ on $\e\partial I_n^\w$, and recalling that $\psi_E^\w+E (\cdot-x_n^\w)$ is a rigid motion on $I_n^\w$, the boundary conditions for $u_\e^\w$ on $\e\partial I_n^\w$ lead to
%
\begin{equation*}
\quad\Big|\int_{\e\partial I_n^\w}v_\e^\w\cdot\sigma(u_\e^\w,P_\e^\w)\nu\Big| \, \lesssim\, \e^2\|\nabla^2 \bar v\|_{\Ld^\infty}\max_{E\in\Ec}\int_{\e\partial I_n^\w} (1+|\psi_E^\w(\tfrac\cdot\e)|)\,(|\nabla u_\e^\w|+|P_\e^\w|).
\end{equation*}
Summing over $n$ and using Cauchy-Schwarz' inequality and the estimates~\eqref{eq:est-sum-par-u} and~\eqref{eq:est-sum-par-cor} of Step~3, we obtain
\begin{equation*}
\quad\Big|\sum_{n\in\Nc_\e^\w(U)} \int_{\e\partial I_n^\w}v_\e^\w\cdot \sigma(u_\e^\w,P_\e^\w)\nu\Big| \, \lesssim\, \|\nabla^2 \bar v\|_{\Ld^\infty}
\Big(\max_{E\in\Ec}\int_{U} (\e+|\e\psi_E^\w(\tfrac\cdot\e)|)^2\Big)^\frac12 \Big( \int_{U} |f|^2\Big)^\frac12,
\end{equation*}
where the right-hand side tends to $0$ by the sublinearity of $\psi_E^\w$ at infinity (cf.~Proposition~\ref{prop:corr-Stokes}(iv)).
\end{enumerate}
Inserting the above estimates into~\eqref{eq:test-osc}, the claim~\eqref{eq:oscillating-test-fct-res} follows.

\medskip
\step5 Computation of the limit in~\eqref{eq:oscillating-test-fct-res} by compensated compactness.\\
For all $\bar v\in C^\infty_c(U)^d$ with $\Div\bar v=0$, we claim that for almost all $\omega$,
\begin{equation}\label{e.step5}
\lim_{\e \downarrow 0}\sum_{E\in\Ec}2\int_U(\nabla_E\bar v)\,\D(\psi_E^\w)(\tfrac\cdot\e):\D(u_\e^\w)\,=\,\sum_{E\in\Ec}\expec{Z_E}:\int_{U}\bar u^\w\otimes\nabla\nabla_E\bar v,
\end{equation}
in terms of the (matrix-valued) stationary random field $Z_{E}$ defined componentwise by
\begin{equation}\label{eq:def-ZE}
Z_{E}^\w:=-\sum_n \frac{\mathds1_{I_n^\w}}{|I_n^\w|}\int_{\partial I_n^\w} \sigma( \psi_E^\w+Ex,\Sigma_E^\w)\nu\otimes(x-x_n^\w).
\end{equation}
Integrating by parts, using equation~\eqref{eq:corr-re} for the corrector, and the constraint $\Div u_\e^\w=0$, we may rewrite the product $\D(\psi_E^\w)(\tfrac\cdot\e):\D(u_\e^\w)$ of two weakly convergent sequences as 
\begin{multline*}
2\int_U(\nabla_E\bar v)\D(\psi_E^\w)(\tfrac\cdot\e):\D(u_\e^\w)
\,=\,-\int_{U\setminus\e\Ic^\w}(u_\e^\w\otimes\nabla\nabla_E\bar v):\big(2\D(\psi_E^\w)-\Sigma_E^\w\big)(\tfrac\cdot\e)\\
-\sum_n\int_{U\cap\e\partial I_n^\w}(\nabla_E\bar v)\,u_\e^\w\cdot\sigma\big(\e\psi_E^\w(\tfrac\cdot\e)+Ex,\Sigma_E^\w(\tfrac\cdot\e)\big)\nu.
\end{multline*}
By the ergodic theorem in form of $\nabla\psi_E^\w(\tfrac\cdot\e)\cvf0$ and $\Sigma_E^\w(\tfrac\cdot\e)\cvf0$ in $\Ld^2(U)$ (cf.~Proposition~\ref{prop:corr-Stokes}(iii)) and by the strong convergence $u_\e^\w\to\bar u^\w$ in $\Ld^2(U)$, the first right-hand side term converges to $0$. Hence, the limit of interest (which exists by~\eqref{eq:oscillating-test-fct-res}) takes the form
\begin{equation}\label{e.intermed}
L^\w:=\lim_{\e \downarrow 0}\sum_{E\in\Ec}2\int_U(\nabla_E\bar v)\,\D(\psi_E^\w)(\tfrac\cdot\e):\D(u_\e^\w)\,=\,\lim_{\e \downarrow 0}I_\e^\w,
\end{equation}
where $I_\e^\w$ denotes the third  and main right-hand side term in the above,
\[I_\e^\w:=-\sum_{E\in\Ec}\sum_n\int_{U\cap\e\partial I_n^\w}(\nabla_E\bar v)\,u_\e^\w\cdot\sigma\big(\e\psi_E^\w(\tfrac\cdot\e)+Ex,\Sigma_E^\w(\tfrac\cdot\e)\big)\nu.\]
Since $\bar v$ is compactly supported in $U$, we may restrict to $\e$ small enough such that $\bar v$ is supported in $\{x\in U:d(x,\partial U)>\e\}$, so that $\bar v$ vanishes on $U\cap\e\partial I_n^\w$ for $n\notin\Nc_\e^\w(U)$. The above thus becomes
\[I_\e^\w=-\sum_{E\in\Ec}\sum_{n\in\Nc_\e^\w(U)}\int_{\e\partial I_n^\w}(\nabla_E\bar v)\, u_\e^\w\cdot\sigma\big(\e\psi_E^\w(\tfrac\cdot\e)+Ex,\Sigma_E^\w(\tfrac\cdot\e)\big)\nu.\]
For $n\in\Nc_\e^\w(U)$, since $u_\e^\w$ is a rigid motion in $\e I_n^\w$, the boundary conditions for the corrector~$\psi_E^\w$ on $\partial I_n^\w$ ensure that
\[\int_{\e\partial I_n^\w}u_\e^\w\cdot\sigma\big(\e\psi_E^\w(\tfrac\cdot\e)+Ex,\Sigma_E^\w(\tfrac\cdot\e)\big)\nu=0,\]
which allows to  reformulate $I_{\e}^\w$ as
\begin{equation}\label{eq:id-0-repl}
I_{\e}^\w\,=\,-\sum_{E\in\Ec}\sum_{n\in\Nc_\e^\w(U)}\int_{\e\partial I_n^\w}\big(\nabla_E\bar v-\nabla_E\bar v(\e x_n^\w)\big)
\,u_\e^\w
\cdot\sigma\big(\e\psi_E^\w(\tfrac\cdot\e)+Ex,\Sigma_E^\w(\tfrac\cdot\e)\big)\nu.
\end{equation}
For $n\in\Nc_\e^\w(U)$, since $u_\e^\w$ is affine in $\e I_n^\w$, we can write on $\e\partial I_n^\w$,
\begin{equation*}
u_\e^\w=\Big(\fint_{\e I_n^\w}u_\e^\w\Big)+(x-\e x_n^\w)_i\nabla_i u_\e^\w,
\end{equation*}
so that
\begin{multline}\label{eq:id-1-repl}
\Big|\big(\nabla_E\bar v-\nabla_E\bar v(\e x_n^\w)\big)u_\e^\omega-(x-\e x_n^\w)_i\Big(\fint_{\e I_n^\w}u_\e^\w\,\nabla_i\nabla_E\bar v \Big)\Big|\\
\lesssim\e^2(|u_\e^\w|+|\nabla u_\e^\w|)\big(\|\nabla^3\bar v\|_{\Ld^\infty(U)}+\|\nabla^2\bar v\|_{\Ld^\infty(U)}\big).
\end{multline}
Next, appealing to the estimates~\eqref{eq:est-sum-par-pu}, \eqref{eq:est-sum-par-u}, and~\eqref{eq:est-sum-par-cor+} of Step~3, we obtain
\begin{eqnarray}
\lefteqn{\sum_{n\in\Nc_\e^\w(U)}\e^2\int_{\e\partial I_n^\w}\big|\sigma\big(\e\psi_E^\w(\tfrac\cdot\e)+Ex,\Sigma_E^\w(\tfrac\cdot\e)\big)\big| \big(|u_\e^\w|+|\nabla u_\e^\w|\big)}\nonumber\\
&\lesssim&\e\bigg(\sum_{n\in\Nc_\e^\w(U)}\e\int_{\e\partial I_n^\w}|\nabla\psi_E^\w(\tfrac\cdot\e)+E|^2+|\Sigma_E^\w(\tfrac\cdot\e)|^2\bigg)^\frac12\bigg(\sum_{n\in\Nc_\e^\w(U)}\e\int_{\e\partial I_n^\w}|u_\e^\w|^2+|\nabla u_\e^\w|^2\bigg)^{\frac12}\nonumber\\
&\lesssim_\delta&\e\, \|f\|_{\Ld^2(U)}\Big(\int_{U}|E|^2+|\nabla\psi_E^\w(\tfrac\cdot\e)|^2+|(\Sigma_E^\w\mathds1_{\R^d\setminus\Ic^\w})(\tfrac\cdot\e)|^2\Big)^\frac12.\label{eq:est-err-loc}
\end{eqnarray}
Inserting~\eqref{eq:id-1-repl} into~\eqref{e.intermed} and~\eqref{eq:id-0-repl}, and using the above to estimate the errors together with the boundedness statement of Proposition~\ref{prop:corr-Stokes}(iii), we are led to
\begin{multline*}
L^\w=\lim_{\e\downarrow0}I_{\e}^\w\,=\,-\lim_{\e\downarrow0}\sum_{E\in\Ec}\sum_{n\in\Nc_\e^\w(U)}\Big(\fint_{\e I_n^\w}u_\e^\w\,\nabla_i\nabla_E\bar v \Big)\\
\cdot\Big(\int_{\e\partial I_n^\w}(x-\e x_n^\w)_i\,\sigma\big(\e\psi_E^\w(\tfrac\cdot\e)+Ex,\Sigma_E^\w(\tfrac\cdot\e)\big)\nu\Big).
\end{multline*}
Recalling that for $\e$ small enough the test function $\bar v$ vanishes on $\e I_n^\w$ for $n\notin\Nc_\e^\w(U)$,
we can rewrite
\begin{equation*}
L^\w\,=\,\lim_{\e\downarrow0}\sum_{E\in\Ec}\int_{U}(u_\e^\w\otimes\nabla\nabla_E\bar v) : Z_E^\w(\tfrac\cdot\e),
\end{equation*}
in terms of the (matrix-valued) stationary field $Z_E$ defined in~\eqref{eq:def-ZE}.
Since $Z_E$ is stationary and bounded in $\Ld^2(\Omega)$, the ergodic theorem ensures $Z_E^\w(\frac \cdot \e)\cvf\expec{Z_E}$ in $\Ld^2(U)$ for typical~$\w$. Combining this with the strong convergence $u_\e^\w\to\bar u^\w$ in $\Ld^2(U)$,
the claim~\eqref{e.step5} follows.

\medskip
\step6 Identification of $\expec{Z_E}$: for all $E\in\Md_0^\Sym$,
\begin{equation}\label{eq:EZE}
\expec{Z_E}~=~2(\Id-\Bb)E-(\bb:E)\Id,
\end{equation}
where $\Bb$ and $\bb$ are defined in~\eqref{eq:def-B} and~\eqref{eq:def-b}.

\medskip\noindent
Let $E\in\Md_0^\Sym$ be fixed.
First note that for any skew-symmetric matrix $E'\in\Md^\Skew$ the definition of $Z_E$ and the boundary conditions for $\psi_E$ entail $E':Z_E=0$ almost surely.
Also note that the definition~\eqref{eq:def-b} of $\bb$ takes the form $\bb:E=-\frac1d\expec{\Tr Z_E}$.
It then suffices to prove~\eqref{eq:EZE} when testing with symmetric trace-free matrices, that is, for all $E'\in\Md^\Sym_0$,
\begin{equation}\label{eq:EZE-toprove}
E':\expec{Z_E}\,=\,E':2(\Id-\Bb)E.
\end{equation}
Let $E'\in\Md^\Sym_0$ be fixed.
For $\eta>0$, choose a cut-off function $\chi_\eta\in C^\infty_c(B)$ with $0\le\chi_{\eta}\le1$ pointwise, with $\chi_{\eta}=1$ on $B_{1-\eta}$, and with $|\nabla\chi_{\eta}|\lesssim\frac1\eta$.
For $0<\e<\frac14\eta$, in view of the hardcore condition, we can construct a modification $\chi_{\e,\eta}^\w\in C^\infty_c(B)$ of $\chi_\eta$ that satisfies the same properties as $\chi_\eta$, such that in addition $\chi_{\e,\eta}^\w$ is constant in each inclusion of the collection $\{\e I_n^\w\}_n$, vanishes in inclusions $\e I_n^\w$ with $n\notin\Nc_\e(B)$, and such that $\chi_{\e,\eta}^\w\to\chi_\eta$ in $\Ld^\infty(B)$ as $\e\downarrow0$.
The ergodic theorem yields for almost all $\w$,
\[E':\expec{Z_E}=\lim_{\eta\downarrow0}\lim_{\e\downarrow0}\fint_B\chi_{\e,\eta}^\w (E':Z_E^\w)(\tfrac\cdot\e).\]
Injecting the definition~\eqref{eq:def-ZE} of $Z_E$ yields
\begin{multline*}
E':\expec{Z_E}\\
\,=\, -\frac1{|B|}\lim_{\eta\downarrow0}\lim_{\e\downarrow0}\sum_{n\in\Nc_\e^\w(B)}\Big(\fint_{\e I_n^\w}\chi_{\e,\eta}^\w \Big)\Big(\int_{\e\partial I_n^\w}E'(x-\e x_n^\w)\cdot\sigma\big(\e\psi_E^\w(\tfrac\cdot\e)+Ex,\Sigma_E^\w(\tfrac\cdot\e)\big)\nu\Big).
\end{multline*}
Since for all $n$ the corrector $\psi_{E'}^\w$ has the form $\kappa_n^\w+\Theta_n^\w(x-x_n^\w)-E'(x-x_n^\w)$ on $I_n^\w$ for some $\kappa_n^\w\in\R^d$ and $\Theta_n^\w\in\Md^\Skew$, the boundary conditions for $\psi_E^\w$ on $\partial I_n^\w$ allow to rewrite
\begin{equation*}
E':\expec{Z_E}\,=\,\frac1{|B|}\lim_{\eta\downarrow0}\lim_{\e\downarrow0}\sum_{n\in\Nc_\e(B)}\Big(\fint_{\e I_n^\w}\chi_{\e,\eta}^\w \Big)\Big(\e\int_{\e\partial I_n^\w}\psi_{E'}^\w(\tfrac\cdot\e)\cdot\sigma\big(\e\psi_E^\w(\tfrac\cdot\e)+Ex,\Sigma_E^\w(\tfrac\cdot\e)\big)\nu\Big).
\end{equation*}
Since $\chi_{\e,\eta}^\w$ is constant in each inclusion and vanishes in inclusions $\e I_n^\w$ with $n\notin \Nc_\e(B)$, this is equivalently written as
\begin{equation*}
E':\expec{Z_E}\,=\,\frac1{|B|}\lim_{\eta\downarrow0}\lim_{\e\downarrow0}\sum_{n}\e\int_{\e\partial I_n^\w}\chi_{\e,\eta}^\w\,\psi_{E'}^\w(\tfrac\cdot\e)\cdot\sigma\big(\e\psi_E^\w(\tfrac\cdot\e)+Ex,\Sigma_E^\w(\tfrac\cdot\e)\big)\nu.
\end{equation*}
Using equation~\eqref{eq:corr-re} for the corrector $\psi_E^\w$ together with $\Div\psi_{E'}^\w=0$, in form of
\begin{multline*}
\sum_n\e\int _{\e\partial I_n^\w}\chi_{\e,\eta}^\w\,\psi_{E'}^\w(\tfrac\cdot\e)\cdot\sigma\big(\e\psi_E^\w(\tfrac\cdot\e)+Ex,\Sigma_E^\w(\tfrac\cdot\e)\big)\nu\\
=-\int_B\chi_{\e,\eta}^\w\,\nabla\psi_{E'}^\w(\tfrac\cdot\e):\nabla\psi_E^\w(\tfrac\cdot\e)-\int_B\e\psi_{E'}^\w(\tfrac\cdot\e)\otimes\nabla\chi_{\e,\eta}^\w:(\nabla\psi_E^\w-\Sigma_E^\w\Id\mathds1_{\R^d\setminus\Ic^\w})(\tfrac\cdot\e),
\end{multline*}
and noting that the second right-hand side term converges to $0$ as $\e\downarrow0$ in view of the sublinearity of $\psi_{E'}$ (cf.~Proposition~\ref{prop:corr-Stokes}(iv)) and in view of the boundedness statement of Proposition~\ref{prop:corr-Stokes}(iii), we deduce
\begin{equation*}
E':\expec{Z_E}\,=\,-\lim_{\eta\downarrow0}\lim_{\e\downarrow0}\fint_B\chi_{\e,\eta}^\w\,\nabla\psi_{E'}^\w(\tfrac\cdot\e):\nabla\psi_E^\w(\tfrac\cdot\e).
\end{equation*}
Equivalently, again integrating by parts and using that $\Div\psi_{E'}^\w=0$, we have
\begin{equation*}
E':\expec{Z_E}\,=\,-\lim_{\eta\downarrow0}\lim_{\e\downarrow0}2\fint_B\chi_{\e,\eta}^\w\,\D(\psi_{E'}^\w)(\tfrac\cdot\e):\D(\psi_E^\w)(\tfrac\cdot\e).
\end{equation*}
In view of the ergodic theorem and of the definition~\eqref{eq:def-B} of $\Bb$, this yields the claim~\eqref{eq:EZE-toprove} in form of
\begin{equation}\label{eq:ident-ZE-1}
E':\expec{Z_E}\,=\,-2\,\expec{\D(\psi_{E'}):\D(\psi_E)}\,=\,E':2(\Id-\Bb)E.
\end{equation}

\medskip
\step7  Conclusion: convergence result.\\
Combining the results of Steps~4--6, we conclude that for almost all $\w$ there holds $u_\e^\w\cvf\bar u^\w$ weakly in $H^1_0(U)$ as $\e\downarrow0$ along a subsequence, where the limit $\bar u^\w$ satisfies $\Div\bar u^\w=0$ and, for all $\bar v\in C^\infty_c(U)^d$ with $\Div\bar v=0$,
 \[\int_U \D(\bar v) : 2\Bb \D(\bar u^\w)
=(1-\lambda)\int_U\bar v\cdot f,\]
where we recall that $\Bb$ is defined in~\eqref{eq:def-B}.
Note that $\Bb$ is positive definite on $\Md_0^\Sym$:
by linearity of the corrector $E\mapsto \psi_E$ with $\expec{\nabla\psi_E}=0$, we compute for all $E\in\Md_0^\Sym$,
\[E:\Bb E \,=\, |E|^2+\expec{|\nabla\psi_{E}|^2} \,\ge\, |E|^2.\]
Hence,  $\bar u^\w\in H^1_0(U)$ is a weak solution of the following well-posed steady Stokes equation in $U$,
 \[-\Div2\Bb\D(\bar u^\w)+\nabla\bar P^\w=(1-\lambda)f,\qquad\Div\bar u^\w=0,\]
in the sense of \cite[Definition~IV.1.1]{Galdi}.  
In addition, by~\cite[Lemma~IV.1.1]{Galdi},  there exists a unique pressure field $\bar P^\w \in \Ld^2(U)$
with $\int_{U}\bar P^\w=0$ such that this equation holds in the usual weak sense.
By uniqueness for the above problem (e.g.~\cite[Theorem~IV.1.1]{Galdi}), the solution $(\bar u^\w,\bar P^\w)=(\bar u,\bar P)$ is independent of $\w$ and the whole sequence converges. The convergence for the pressure field follows from the corrector result below combined with an approximation argument, cf.~Substep~8.4.

\medskip
\step8 Corrector results.\\
We finally turn to the additional corrector results, which we obtain by a suitable recycling of the above computations.
We consider the following two-scale expansion errors,
\begin{eqnarray*}
w_\e^\w&:=&u_\e^\w-\bar u-\e\sum_{E\in\Ec}\psi_E^\w(\tfrac\cdot\e)\nabla_E\bar u,\\
Q_\e^\w&:=&P_\e^\w\mathds1_{U\setminus\Ic_\e^\w(U)}-\bar P-\bb:\D(\bar u)-(\Sigma_E^\w\mathds1_{\R^d\setminus\e\Ic^\w})(\tfrac\cdot\e)\nabla_E\bar u,
\end{eqnarray*}
and we split the proof into four further substeps: we start with the short proof of the corrector result for the velocity field, that is, $w_\e^\w\to0$ in $H^1(U)$, based on the convergence of the energy, and then we establish a suitable equation for $w_\e^\w$, from which we deduce a bound on the pressure $Q_\e^\w$ and the corresponding corrector result.
In the first three substeps, we assume for simplicity that the homogenized solution $\bar u$ belongs to $W^{3,\infty}(U)^d$, an assumption 
that we relax in the   last substep.

\medskip
\substep{8.1} Corrector result for the velocity field.\\
First, combining~\eqref{eq:energy} with the strong convergence $u_\e^\w\to\bar u$ in $\Ld^2(U)$ yields  for almost all~$\w$ the convergence of energies in the form
\[2\int_U|\D(u_\e^\w)|^2=\int_U|\nabla u_\e^\w|^2=\int_{U\setminus\Ic_\e^\w(U)}f\cdot u_\e^\w\to(1-\lambda)\int_{U}f\cdot\bar u=2\int_U\D(\bar u):\Bb\D(\bar u).\]
Second, using the constraint $\Tr\nabla\bar u=\Div\bar u=0$ in the form $\D(\bar u)=\sum_{E\in\Ec}(\nabla_E\bar u)E$, and appealing to the stationarity of $\nabla\psi_E$, the ergodic theorem, and the sublinearity of $\psi_E$ (cf.~Proposition~\ref{prop:corr-Stokes}(iv)),  together with the additional regularity of $\bar u$, we find for almost all~$\w$,
\begin{multline*}
\int_U\Big|\D\Big(\bar u+\e\sum_{E\in\Ec}\psi_E^\w(\tfrac\cdot\e)\nabla_E\bar u\Big)\Big|^2
=\int_U\Big|\sum_{E\in\Ec}(\D(\psi_E^\w)+E)(\tfrac\cdot\e)\nabla_E\bar u+\sum_{E\in\Ec}\big(\e\psi_E^\w(\tfrac\cdot\e)\otimes\nabla\nabla_E\bar u\big)^s\Big|^2\\
\to\int_U\D(\bar u):\Bb\D(\bar u).
\end{multline*}
Third, choosing $\bar v=\bar u$ as a test function, the computations of Steps~5--6  together with the regularity of $\bar u$ precisely yield for almost all $\w$,
\begin{multline*}
\int_U\D\Big(\bar u+\e\sum_{E\in\Ec}\psi_E^\w(\tfrac\cdot\e)\nabla_E\bar u\Big):\D( u_\e^\w)\\
=\sum_{E\in\Ec}\int_U(\nabla_E\bar u)\,(\D(\psi_E^\w)+E)(\tfrac\cdot\e):\D(u_\e^\w)+\sum_{E\in\Ec}\int_U\e\psi_E^\w(\tfrac\cdot\e)\otimes\nabla\nabla_E\bar u:\D( u_\e^\w)\\
\to\int_U\D(\bar u):\Bb\D(\bar u).
\end{multline*}
Combining the above and reconstructing the square lead to the stated corrector result for the velocity field.

\medskip

\substep{8.2} Equation for the two-scale expansion error.\\
We claim that $(w_\e^\w,Q_\e^\w)$ satisfies in the weak sense in $U$,
\begin{multline}\label{eq:error-pde}
-\triangle w_\e^\w+\nabla Q_\e^\w=-\sum_{n\in\Nc_\e^\w(U)}\delta_{\e\partial I_n^\w}\sigma(u_\e^\w,P_\e^\w)\nu\\
-\Div2(\Bb-\Id)\D(\bar u)-\nabla(\bb:\D(\bar u))+\sum_{E\in\Ec}\nabla_E\bar u\sum_n\delta_{\e\partial I_n^\w}\sigma\big(\e\psi_E^\w(\tfrac\cdot\e)+Ex,\Sigma_E^\w(\tfrac\cdot\e)\big)\nu\\
+(\lambda-\mathds1_{\Ic_\e^\w(U)})f-\sum_{E\in\Ec}(\Sigma_E^\w\mathds1_{\R^d\setminus\e\Ic^\w})(\tfrac\cdot\e)\nabla\nabla_E\bar u+\Div\Big(\sum_{E\in\Ec}\e\psi_E^\w(\tfrac\cdot\e)\otimes\nabla\nabla_E\bar u\Big).
\end{multline}
Combining equations~\eqref{eq:Stokes-re} and~\eqref{eq:corr-re} indeed yields
\begin{multline*}
-\triangle w_\e^\w+\nabla Q_\e^\w=\triangle\bar u-\nabla\bar P-\nabla(\bb:\D(\bar u))+f\mathds1_{U\setminus\Ic_\e^\w(U)}-\sum_{n\in\Nc_\e^\w(U)}\delta_{\e\partial I_n^\w}\sigma(u_\e^\w,P_\e^\w)\nu\\
+\sum_{E\in\Ec}\nabla_E\bar u\sum_n\delta_{\e\partial I_n^\w}\sigma\big(\e\psi_E^\w(\tfrac\cdot\e)+Ex,\Sigma_E^\w(\tfrac\cdot\e)\big)\nu\\
-\sum_{E\in\Ec}(\Sigma_E^\w\mathds1_{\R^d\setminus\Ic^\w})(\tfrac\cdot\e)\nabla\nabla_E\bar u+\sum_{E\in\Ec}\nabla\psi_E^\w(\tfrac\cdot\e)\nabla\nabla_E\bar u+\e\sum_{E\in\Ec}\psi_E^\w(\tfrac\cdot\e)\triangle\nabla_E\bar u,
\end{multline*}
and the claim follows after inserting the equation for $\bar u$ and recombining the last two right-hand side terms.

\medskip
\substep{8.3} Corrector result for the pressure field.\\
As in Substep~4.2 of the proof of Proposition~\ref{prop:corr-Stokes}, for almost all $\w$, we can construct a map $\zeta_\e^\w\in H^1_0(U)$ such that $\zeta_\e^\w|_{\e I_n^\w}$ is constant for all $n\in\Nc_\e^\w(U)$ and such that
\begin{gather}
\Div\zeta_\e^\w=\Big(Q_\e^\w-\fint_{U\setminus\Ic_\e^\w(U)}Q_\e^\w\Big)\mathds1_{U\setminus\Ic_\e^\w(U)},\label{eq:def-zeta-eps1}\\
\|\nabla\zeta_\e^\w\|_{\Ld^2(U)}\lesssim\Big\|Q_\e^\w-\fint_{U\setminus\Ic_\e^\w(U)}Q_\e^\w\Big\|_{\Ld^2(U\setminus\Ic_\e^\w(U))}.\label{eq:def-zeta-eps2}
\end{gather}
Testing equation~\eqref{eq:error-pde} with $\zeta_\e^\w$, using the boundary conditions for $u_\e^\w$ at inclusion boundaries, and recalling that $\zeta_\e^\w$ is constant on each inclusion $\e I_n^\w$ with $n\in\Nc_\e^\w(U)$, we find
\begin{multline*}
\int_U\nabla\zeta_\e^\w:\nabla w_\e^\w-\int_U(\Div\zeta_\e^\w)\, Q_\e^\w=
\int_U\D(\zeta_\e^\w):2(\Bb-\Id)\D(\bar u)+\int_U(\Div \zeta_\e^\w)\,\bb:\D(\bar u)\\
+\sum_{E\in\Ec}\sum_{n}\int_{\e\partial I_n^\w}(\nabla_E\bar u)\,\zeta_\e^\w\cdot\sigma\big(\e \psi_E^\w(\tfrac\cdot\e)+Ex,\Sigma_E^\w(\tfrac\cdot\e)\big)\nu\\
+\int_U(\lambda-\mathds1_{\Ic_\e^\w(U)})\,\zeta_\e^\w\cdot f-\sum_{E\in\Ec}\int_{U}(\Sigma_E^\w\mathds1_{\R^d\setminus\Ic^\w})(\tfrac\cdot\e)\,\zeta_\e^\w\cdot\nabla\nabla_E\bar u-\int_U\nabla\zeta_\e^\w:\sum_{E\in\Ec}\e\psi_E^\w(\tfrac\cdot\e)\otimes\nabla\nabla_E\bar u.
\end{multline*}
In view of properties~\eqref{eq:def-zeta-eps1} and~\eqref{eq:def-zeta-eps2} of the test function $\zeta_\e^\w$, we deduce after reorganizing the terms,
\begin{equation}\label{eq:bound-press-q}
\int_{U\setminus\Ic_\e^\w(U)}\Big(Q_\e^\w-\fint_{U\setminus\Ic_\e^\w(U)}Q_\e^\w\Big)^2
\,\lesssim\,\sum_{j=1}^5T_{\e,j}^\w,
\end{equation}
where
\begingroup\allowdisplaybreaks
\begin{eqnarray*}
T_{\e,1}^\w&:=&\int_U|\nabla w_\e^\w|^2,\\
T_{\e,2}^\w&:=&\Big|\int_U(\lambda-\mathds1_{\Ic_\e^\w(U)})\,\zeta_\e^\w\cdot f\Big|+\sum_{E\in\Ec}\Big|\int_{U}(\Sigma_E^\w\mathds1_{\R^d\setminus\e\Ic^\w})(\tfrac\cdot\e)\,\zeta_\e^\w\cdot\nabla\nabla_E\bar u\Big|,\\
T_{\e,3}^\w&:=&\sum_{E\in\Ec}\int_U\e|\psi_E^\w(\tfrac\cdot\e)||\nabla\zeta_\e^\w||\nabla\nabla_E\bar u|,\\
T_{\e,4}^\w&:=&\sum_{E\in\Ec}\sum_{n\notin\Nc_\e^\w(U)}\int_{\e\partial I_n^\w}|\nabla_E\bar u|\,|\zeta_\e^\w|\big|\sigma\big(\e\psi_E^\w(\tfrac\cdot\e)+Ex,\Sigma_E^\w(\tfrac\cdot\e)\big)\big|,\\
T_{\e,5}^\w&:=&\bigg|\int_U\D(\zeta_\e^\w):2(\Bb-\Id)\D(\bar u)+\int_U(\Div\zeta_\e^\w)\,\bb:\D(\bar u)\\
&&\qquad+\sum_{E\in\Ec}\sum_{n\in\Nc_\e^\w(U)}\int_{\e\partial I_n^\w}(\nabla_E\bar u)\,\zeta_\e^\w\cdot\sigma\big(\e\psi_E^\w(\tfrac\cdot\e)+Ex,\Sigma_E^\w(\tfrac\cdot\e)\big)\nu\bigg|.
\end{eqnarray*}
\endgroup
We successively estimate these different terms.
First, the corrector result for the velocity field in Step~8.1 yields $T_{\e,1}^\w\to0$ for almost all $\w$.
We turn to the second term $T_{\e,2}^\w$.
In view of~\eqref{e.energy-estim} and of the boundedness statement of Proposition~\ref{prop:corr-Stokes}(iii), with the regularity of $\bar u$, we deduce that for almost all $\w$ the pressure $Q_\e^\w$ is bounded in $\Ld^2(U)$ uniformly in~$\e$, hence in view of~\eqref{eq:def-zeta-eps2} the test function $\zeta_\e^\w$ is bounded in $H^1_0(U)$.
For almost all~$\w$, by weak compactness, there exists $\bar\zeta^\w\in H^1_0(U)^d$ such that $\zeta_\e^\w\cvf\bar\zeta^\w$ in~$H^1_0(U)$ along some subsequence (not relabelled), hence also $\zeta_\e^\w\to\bar\zeta^\w$ in $\Ld^2(U)$ by Rellich's theorem. Combining this strong convergence with the ergodic theorem in form of $(\Sigma_E^\w\mathds1_{\R^d\setminus\e\Ic^\w})(\tfrac\cdot\e)\cvf0$ in $\Ld^2(U)$ (cf.~Proposition~\ref{prop:corr-Stokes}(iii)) and in form of $\mathds1_{\Ic_\e^\w(U)}\cvf\lambda\mathds1_U$ weakly-* in $\Ld^\infty(U)$, together with the regularity of $\bar u$, we deduce $T_{\e,2}^\w\to0$ for almost all $\w$.
Similarly, in view of the sublinearity of~$\psi_E$ (cf.~Proposition~\ref{prop:corr-Stokes}(ii)), we find $T_{\e,3}^\w\to0$.

\medskip\noindent
We turn to the boundary term $T_{\e,4}^\w$.
For $n\notin\Nc_\e^\w(U)$ with $\e I_n^\w\cap U\ne\varnothing$, since $\e\partial I_n^\w$ is at distance at most $\e$ from $\partial U$, on which $\zeta_\e^\w$ vanishes, we deduce from a trace estimate,
\[\|\zeta_\e^\w\|_{\Ld^2(\e\partial I_n^\w\cap U)}\,\lesssim\,\e^\frac12\|\nabla\zeta_\e^\w\|_{\Ld^2(\e (I_n^\w+ B)\cap U)},\]
hence by Cauchy-Schwarz' inequality,
\[T_{\e,4}^\w\,\lesssim\,\|\nabla\bar u\|_{\Ld^\infty(U)}\|\nabla\zeta_\e^\w\|_{\Ld^2(U)}\sum_{E\in\Ec}\Big(\sum_{n\notin\Nc_\e^\w(U)}\e\int_{\e\partial I_n^\w}\big|\sigma\big(\e\psi_E(\tfrac\cdot\e)+Ex,\Sigma_E^\w(\tfrac\cdot\e)\big)\big|^2\Big)^\frac12.\]
As in the proof of~\eqref{eq:est-sum-par-cor+}, we appeal to a trace estimate and to the regularity theory for the Stokes equation near a boundary in the form
\[\int_{\partial I_n^\w}|\nabla\psi_E^\w|^2+|\Sigma_E^\w|^2\,\lesssim\,\|(\nabla\psi_E^\w,\Sigma_E^\w)\|_{\Ld^2(I_n^\w+\frac\delta2 B)}^2,\]
so that the above yields
\[T_{\e,4}^\w\,\lesssim\,\|\nabla\bar u\|_{\Ld^\infty(U)}\|\nabla\zeta_\e^\w\|_{\Ld^2(U)}\sum_{E\in\Ec}\Big(\int_{(\partial U)+3\e B}|E|^2+|\nabla\psi_E^\w(\tfrac\cdot\e)|^2+|(\Sigma_E^\w\mathds1_{\R^d\setminus\Ic^\w})(\tfrac\cdot\e)|^2\Big)^\frac12,\]
where the right-hand side  converges to $0$ for almost all $\w$ as a consequence of the ergodic theorem, cf. Proposition~\ref{prop:corr-Stokes}(iii).

\medskip\noindent
It remains to estimate $T_{\e,5}^\w$, and we use the short-hand notation
\[J_\e^\w:=-\sum_{E\in\Ec}\sum_{n\in\Nc_\e(U)}\int_{\e\partial I_n^\w}(\nabla_E\bar u)\,\zeta_\e^\w\cdot\sigma\big(\e\psi_E^\w(\tfrac\cdot\e)+Ex,\Sigma_E^\w(\tfrac\cdot\e)\big)\nu.\]
As shown in Step~5, in view of the boundary conditions for $\psi_E$ at the inclusion boundaries, together with the regularity of $\bar u$, an approximation argument for $\nabla_E\bar u$ leads to
\begin{multline}\label{eq:approx-Jeps}
\lim_{\e\downarrow0}\bigg|J_\e^\w
+\sum_{E\in\Ec}\sum_{n\in\Nc_\e^\w(U)}\Big(\fint_{\e I_n^\w}\zeta_\e^\w\nabla_i\nabla_E\bar u\Big)\\\cdot\Big(\int_{\e\partial I_n^\w}(x-\e x_n^\w)_i\,\sigma\big(\e\psi_E^\w(\tfrac\cdot\e)+Ex,\Sigma_E^\w(\tfrac\cdot\e)\big)\nu\Big)\bigg|\,=\,0.
\end{multline}
Applying Cauchy-Schwarz' inequality in the form
\begin{multline*}
\sum_{n\notin\Nc_\e^\w(U)}\bigg|\Big(\fint_{\e I_n^\w}\zeta_\e^\w\nabla_i\nabla_E\bar u\Big)\cdot\Big(\int_{\e\partial I_n^\w}(x-\e x_n^\w)_i\,\sigma\big(\e\psi_E^\w(\tfrac\cdot\e)+Ex,\Sigma_E^\w(\tfrac\cdot\e)\big)\nu\Big)\bigg|\\
\lesssim\,\|\nabla^2\bar u\|_{\Ld^\infty(U)}\|\zeta_\e^\w\|_{\Ld^2(U)}\Big(\sum_{n\notin\Nc_\e^\w(U)}\e\int_{\e\partial I_n^\w}\big|\sigma\big(\e\psi_E^\w(\tfrac\cdot\e)+Ex,\Sigma_E^\w(\tfrac\cdot\e)\big)\big|^2\Big)^\frac12,
\end{multline*}
and noting that the estimate on $T_{\e,4}^\w$ above ensures that the right-hand side converges to~$0$, we deduce that the restriction to $n\in\Nc_\e^\w(U)$ can be removed from the sum in~\eqref{eq:approx-Jeps}. In terms of the random field $Z_E$ defined in~\eqref{eq:def-ZE}, we are thus led to
\begin{equation*}
\lim_{\e\downarrow0}\Big|J_\e^\w
-\sum_{E\in\Ec}\int_U(\zeta_\e^\w\otimes\nabla\nabla_E\bar u):Z_E^\w(\tfrac\cdot\e)\Big|\,=\,0.
\end{equation*}
Appealing to the ergodic theorem for $Z_E$, to the identification of $\expec{Z_E}$ in Step~6, and to the strong convergence $\zeta_\e^\w\to\bar\zeta^\w$ in $\Ld^2(U)$, together with the regularity of $\bar u$, we deduce
\begin{equation*}
\lim_{\e\downarrow0}\Big|J_\e^\w-\int_U\D(\bar\zeta^\w):2(\Bb-\Id)\D(\bar u)-\int_U(\Div\bar\zeta^\w)\,\bb:\D(\bar u)\Big|\,=\,0,
\end{equation*}
that is, $T_{\e,5}^\w\to0$.
We conclude that the whole right-hand side in~\eqref{eq:bound-press-q} converges to $0$ for almost all~$\w$, and the corrector result follows.

\medskip
\substep{8.4} Relaxing the regularity assumption.\\
Assume that $f\in\Ld^p(U)$ for some $p>d$ and note that in view of the regularity theory for the homogenized Stokes equation~\eqref{eq:Stokes-hom} in form of~\cite[Lemma~IV.6.1]{Galdi} this implies $\bar u\in W^{2,p}_0(U)^d$ and $\bar P\in W^{1,p}(U)$. Choosing an approximating sequence $(f^r)_r\subset C^\infty_b(U)$ with $f^r\to f$ in $\Ld^p(U)$ as $r\downarrow0$, we deduce by linearity that the corresponding solution $(\bar u^r,\bar P^r)$ of the homogenized equation satisfies $\bar u^r\to\bar u$ in $W^{2,p}(U)$, hence $\bar u^r\to\bar u$ in $W^{1,\infty}\cap W^{2,d}(U)$ by the Sobolev embedding.
In addition, in view of the energy estimate~\eqref{e.energy-estim}, the corresponding solution $(u_\e^{r,\w},P_\e^{r,\w})$ of~\eqref{eq:Stokes} satisfies
\begin{equation}\label{eq:conv-approx}
\sup_{\e>0}\int_U|\nabla(u_\e^{r,\w}-u_\e^\w)|^2+\sup_{\e>0}\int_{U\setminus\Ic_\e^\w(U)}|P_\e^{r,\w}-P_\e^\w|^2\,\lesssim\,\int_U|f^r-f|^2\,\to\,0,
\end{equation}
as $r\downarrow0$.
Since for fixed $r>0$ the approximation $f^r$ is smooth, the regularity theory for the homogenized Stokes equation ensures that $\bar u^r$ belongs at least to $W^{3,\infty}(U)^d$, hence the above Steps~8.1--8.3 show that the corrector results indeed hold for the $r$-approximations. Since $\nabla\psi_E^\w(\tfrac\cdot\e)$ and $(\Sigma_E^\w\mathds1_{\R^d\setminus\Ic^\w})(\tfrac\cdot\e)$ are bounded in $\Ld^2(U)$ for almost all $\w$ (cf.~Proposition~\ref{prop:corr-Stokes}(iii)), and since the Sobolev embedding also ensures the boundedness of $\e\psi_E^\w(\tfrac\cdot\e)$ in $\Ld^{2d/(d-2)}(U)$, the above convergences precisely allow to get rid of approximations.

\medskip\noindent
We conclude with the argument for the weak convergence of the pressure for $f \in \Ld^2(U)$. Choose an approximating sequence $(f^r)_r\subset C^\infty_b(U)$ with $f^r\to f$ in $\Ld^2(U)$, and denote by $(u_\e^{r,\w},P_\e^{r,\w})$ and by $(\bar u^r,\bar P^r)$ the corresponding solutions of~\eqref{eq:Stokes} and of~\eqref{eq:Stokes-hom}. Starting from the corrector result for the pressure for the regularized data $f^r$, appealing to the ergodic theorem of Proposition~\ref{prop:corr-Stokes}(iii), and noting that the weak-* convergence $\mathds1_{U\setminus\Ic_\e^\w(U)}\cvf1-\lambda$ in~$\Ld^\infty(U)$ entails $\int_U(\bar P^r+\bb:\D(\bar u^r))\mathds1_{U\setminus\Ic_\e^\w(U)}\to(1-\lambda)\int_U(\bar P^r+\bb:\D(\bar u^r))=0$, we obtain for all $r$, for almost all $\w$,
\[(P_\e^{r,\w}-\bar P^r-\bb:\D(\bar u^r))\mathds1_{U\setminus\Ic_\e^\w(U)}~\cvf~0,\quad\text{weakly in $\Ld^2(U)$,\quad as $\e\downarrow0$}.\]
Next, the same argument as above yields $(\bar u^r,\bar P^r)\to(\bar u,\bar P)$ in $H^1_0(U)\times\Ld^2(U)$, as well as~\eqref{eq:conv-approx}, which allows to get rid of approximations in this weak convergence result.
\qed

\section*{Acknowledgements}
MD acknowledges financial support from the CNRS-Momentum program,
and AG from the European Research Council under the European Community's Seventh Framework Programme (FP7/2014-2019 Grant Agreement \textsc{quanthom} 335410).

\bibliographystyle{plain}

\def\cprime{$'$} \def\cprime{$'$} \def\cprime{$'$}


\end{document}